\documentclass[12pt,twoside,reqno,psamsfonts]{amsart}

\usepackage[left=2.3cm,top=3cm,right=2.3cm]{geometry}              
\usepackage[colorlinks = true, citecolor = black, linkcolor = black, urlcolor = black, pdfstartview=FitH]{hyperref}  
\usepackage[pdftex]{graphicx}
\usepackage[utf8]{inputenc}
\usepackage{microtype, float}
\usepackage{amsmath, verbatim}
\usepackage{amssymb}
\usepackage{epstopdf}
\usepackage{epsfig}

\DeclareUnicodeCharacter{00A0}{ } 

\numberwithin{equation}{section}

\theoremstyle{plain}
\newtheorem{thm}{Theorem}[section]

\newtheorem{conj}[thm]{Conjecture}
\newtheorem{lemma}[thm]{Lemma}
\newtheorem{prop}[thm]{Proposition}
\newtheorem{cor}[thm]{Corollary}

\theoremstyle{definition}
\newtheorem{defn}[thm]{Definition}
\newtheorem{definition}[thm]{Definition}

\newtheorem{notation}[thm]{Notation}
\newtheorem*{ack}{Acknowledgements}

\theoremstyle{remark}
\newtheorem{remark}[thm]{Remark}

\newcommand{\R}{\mathbb{R}}
\newcommand{\T}{\mathbb{T}}

\newcommand{\Q}{\mathbb{Q}}
\newcommand{\N}{\mathbb{N}}

\newcommand{\E}{\mathbb{E}}

\newcommand{\cB}{\mathcal{B}}

\newcommand{\cH}{\mathcal{H}}

\newcommand{\cE}{\mathcal{E}}

\newcommand{\cP}{\mathcal{P}}

\renewcommand{\epsilon}{\varepsilon}
\renewcommand{\rho}{\varrho}
\renewcommand{\phi}{\varphi}

\renewcommand{\i}{\mathtt{i}}
\renewcommand{\j}{\mathtt{j}}

\renewcommand{\v}{\mathtt{v}}
\newcommand{\w}{\mathtt{w}}

\renewcommand{\iint}{\int\hspace{-0.1in}\int}

\newcommand{\id}{\operatorname{id}}
\renewcommand{\mod}{\,\,\mathrm{mod}\,}

\DeclareMathOperator{\spt}{spt}

\DeclareMathOperator{\mini}{mini}
\DeclareMathOperator{\micro}{micro}

\usepackage[usenames,dvipsnames]{xcolor}

\begin{document}
	
\title{Scaling scenery of $(\times m,\times n)$ invariant measures}

\author{Andrew Ferguson, Jonathan M. Fraser and Tuomas Sahlsten}

\address{Department of Mathematics, University of Bristol, University Walk, Clifton, BS8 1TW, Bristol, UK}
\email{andrew.ferguson@bristol.ac.uk}

\address{School of Mathematics, The University of Manchester, Manchester, M13 9PL, UK.}
\email{jon.fraser32@gmail.com}

\address{Einstein Institute of Mathematics, The Hebrew University of Jerusalem, Givat Ram, Jerusalem 91904, Israel}
\email{tuomas@sahlsten.org}

\keywords{CP-chain, symbolic dynamics, Bernoulli measure, Hausdorff dimension, self-affine carpet, projections, the distance set conjecture.}
\subjclass[2010]{28D05, 37C45, 28A80}

\thanks{AF acknowledges support from EPSRC grant EP/I024328/1 and the University of Bristol.  JMF acknowledges support from the University of St Andrews, an EPSRC doctoral training grant, the University of Warwick, the EPSRC grant EP/J013560/1 and the University of Manchester.  TS acknowledges the support from the University of Bristol, the Finnish Centre of Excellence in Analysis and Dynamics Research, the Emil Aaltonen Foundation and the European Union (ERC grant $\sharp$306494)}


\begin{abstract}We study the scaling scenery and limit geometry of invariant measures for the non-conformal toral endomorphism $(x,y) \mapsto (mx \mod 1,ny \mod 1)$ that are Bernoulli measures for the natural Markov partition. We show that the statistics of the scaling can be described by an ergodic CP-chain in the sense of Furstenberg. Invoking the machinery of CP-chains yields a projection theorem for Bernoulli measures, which generalises in part earlier results by Hochman-Shmerkin and Ferguson-Jordan-Shmerkin. We also give an ergodic theoretic criterion for the dimension part of Falconer's distance set conjecture for general sets with positive length using CP-chains and hence verify it for various classes of fractals such as self-affine carpets of Bedford-McMullen, Lalley-Gatzouras and Bara\'nski class and all planar self-similar sets.
\end{abstract}

\maketitle


\section{Introduction and main results}
\label{sec:intro}
Using ergodic theory to study problems in geometry is not new, however, there have recently been some major advances in the fields of fractal geometry and geometric measure theory made by studying the dynamics of the process of \textit{magnifying} fractal sets and measures.  In particular, the recent papers of Hochman \cite{Hoc10} and Hochman-Shmerkin \cite{HocShm12} have developed ideas of Furstenberg \cite{furst1, furst2} and introduced a rich theory which has proven most useful in solving many long standing problems in geometry and analysis where scaling dynamics is present. See for example \cite{HocShm13equidist} for applications to equidistribution problems in metric number theory and \cite{jin, HocShm12} for applications to projections of fractal sets and measures.  

The idea of using tangents has been used extensively in the past, for example in the use of tangent measures in \cite{preiss} and in metric geometry in the study of tangent metric spaces. Often it turns out that tangents enjoy more regularity than the original object of study. When the object of study has a conformal structure, for example in the case of self-similar measures satisfying the open set condition, one expects the tangents to be essentially the same object. However, in the presence of non-conformality, in the limit we often obtain a completely different but sometimes more regular geometry.

In order to take advantage of the new more regular limit geometry arising in the non-conformal setting, we have to find a way to transfer information back to the original object. The scaling dynamics of the blow-ups can be modeled using a \textit{CP-chain}, which is a Markov process that records both the point where we zoom-in and the \textit{scenery} which we then see. If the CP-chain enjoys suitable irreducibility properties, the main results of \cite{HocShm12} yield that it is possible to transfer strong geometric information about the \textit{projections} of the limit structures back to the original geometry of interest. In this paper, we find out that a class of non-conformal measures will have scaling statistics that satisfy the assumptions required by Hochman-Shmerkin, which allows us to obtain new geometric results about them.

Our main applications will concern the Hausdorff dimensions of various sets and measures related to classical problems in geometric measure theory.  Throughout the paper we will write $\dim$ for the Hausdorff dimension of a set and for the \textit{(lower) Hausdorff dimension} of a measure, which is defined as
\[
\dim \mu = \inf \{ \dim E : \mu(E) > 0\}.
\]
We will also write $\mathcal{H}^s$ for the $s$-dimensional Hausdorff measure.  For a review of these notions see \cite{Mattila95}.

\subsection{Scenery of $\times m$ and $(\times m,\times n)$ invariants} Important examples of dynamically invariant sets and measures can be found by studying the expanding maps $T_m : \T \to \T$, $m \in \N$, of the unit circle $\T$, where
$$T_m(x) = mx \mod 1.$$
Throughout the paper we will often identify $\T$ and $[0,1)$ (and thus $\T^2$ and $[0,1)^2$)  in the obvious way.  Furstenberg has proposed several conjectures on invariants for these maps, of which perhaps the most famous is the \textit{$\times 2$ $\times3$ conjecture} asking about the uniqueness of simultaneously invariant non-atomic ergodic measures for $T_2$ and $T_3$. In \cite{HocShm12}, among many other things, Hochman and Shmerkin verified the following related conjecture of Furstenberg on projections of products of $T_2$ and $T_3$ invariant sets.  Write $\Pi_{2,1}$ to denote the set of all orthogonal projections $\pi : \R^2 \to \R$ with $\pi_1$ and $\pi_2$ being the horizontal and vertical projections respectively. 

\begin{conj}[Furstenberg] If $X, Y \subset \T$ are closed and invariant under $T_2$ and $T_3$ respectively, then
$$\dim \pi(X \times Y) = \min\{1,\dim (X \times Y)\}$$
for all $\pi \in \Pi_{2,1} \setminus \{\pi_1,\pi_2\}$.
\end{conj}

This provided a sharpening of Marstrand's classical projection theorem in this case, which states that for general $K \subseteq \mathbb{R}^2$ the number $\min\{1,\dim K\}$ is the Hausdorff dimension of the projection $\pi(K)$ for \textit{almost every} orthogonal projection $\pi : \R^2 \to \R$. The result of Hochman-Shmerkin was in fact more general as it was also true for measures:

\begin{thm}[Theorem 1.3 of \cite{HocShm12}] \label{hocshmproj} Suppose $\log m / \log n$ is irrational. If measures $\mu,\nu$ on $\T$ are invariant under $T_m$ and $T_n$ respectively, then
$$\dim \pi(\mu \times \nu) = \min\{1,\dim (\mu \times \nu)\}$$
for all $\pi \in \Pi_{2,1} \setminus \{\pi_1,\pi_2\}$.
\end{thm}

The proof relied on the construction of a CP-chain generated by the product measure $\mu \times \nu$ on the $2$-torus $\T^2$. Notice that by the respective invariance of $\mu$ and $\nu$, the product measure $\mu \times \nu$ is in fact invariant under the non-conformal toral endomorphism $T_{m,n} : \T^2 \to \T^2$, defined by
$$T_{m,n}(x,y) = (mx \mod 1, nx \mod 1).$$
Hence a natural step forward would be to ask if one can extend the results in \cite{HocShm12} on product measures to general $T_{m,n}$ invariant measures. Such a step would be quite a leap, as general $T_{m,n}$ invariant measures can have a much more complicated structure than measures of the product form considered by Hochman and Shmerkin. However, if we restrict our attention to invariant measures which are the push forward of \textit{Bernoulli measures} on the code space via the natural Markov partition, we can construct a CP-chain for them. We will refer to such measures as \textit{Bernoulli measures for} $T_{m,n}$, see Definition \ref{bernoullimeasures} below.  The main result of this paper is as follows.

\begin{thm}
\label{thm:main}
If $\mu$ is a Bernoulli measure for $T_{m,n}$ on $\T^2$, then it generates an ergodic CP-chain.

The measure component of the CP-chain is the distribution on the measures of the form 
$$S_t(\pi_1 \mu \times \mu_x).$$ 
Here $\mu_x$ is the conditional measure of $\mu$ with respect to the vertical fiber $\pi_1^{-1}\{x\}$, where $x$ is drawn according to $\pi_1 \mu$. Moreover, $S_t : \R^2 \to \R^2$ is the hyperbolic matrix
$$S_t = \begin{pmatrix} n^{-t/2} & 0 \\ 0 & n^{t/2} \end{pmatrix},$$ 
where $t \in [0,1)$ is distributed according to Lebesgue if $\frac{\log m}{\log n} \in \R \setminus \Q$, and otherwise according to a periodic measure with respect to $\frac{\log m}{\log n}$-rotation on $\T$, see Definition \ref{ourchain} below.
\end{thm}

The Bernoulli property of $\mu$ in the formulation of Theorem \ref{thm:main} is essential as it provides independence along cylinders for the measure and allows us to create a parametrisation of the blow-ups of $\mu$ at a typical point. Moreover, this parametrisation gives us a clear description of the measures appearing in the CP-chain. Going beyond Bernoulli may be possible if one assumes strong mixing properties for the measure, but we do not pursue this here.  See Section \ref{sec:prospects} for further discussion.

Important examples of $(\times m,\times n)$ invariant sets and measures which are not in general of the product type considered in Theorem \ref{hocshmproj} are the \textit{self-affine carpets} of Bedford and McMullen and the self-affine Bernoulli measures supported on them. Generally, a self-affine set is an attractor of an iterated function system (IFS) where the maps involved are translate linear maps on some Euclidean space.  Recall that an IFS is a finite set of contractions $\{S_i\}_{i=1}^N$, and it is well-known that there exists a unique non-empty compact set $K$ satisfying
\[
K = \bigcup_{i=1}^N S_i(K)
\]
which is called the attractor of the IFS.  Self-affine \textit{carpets} are a special class of self-affine sets.  They have attracted a great deal of attention in the literature in recent years due to their simple construction and the fact that they exhibit many new phenomena not observed in the self-similar setting. 

\begin{figure}[h!]
	\centering
	\includegraphics[width=120mm]{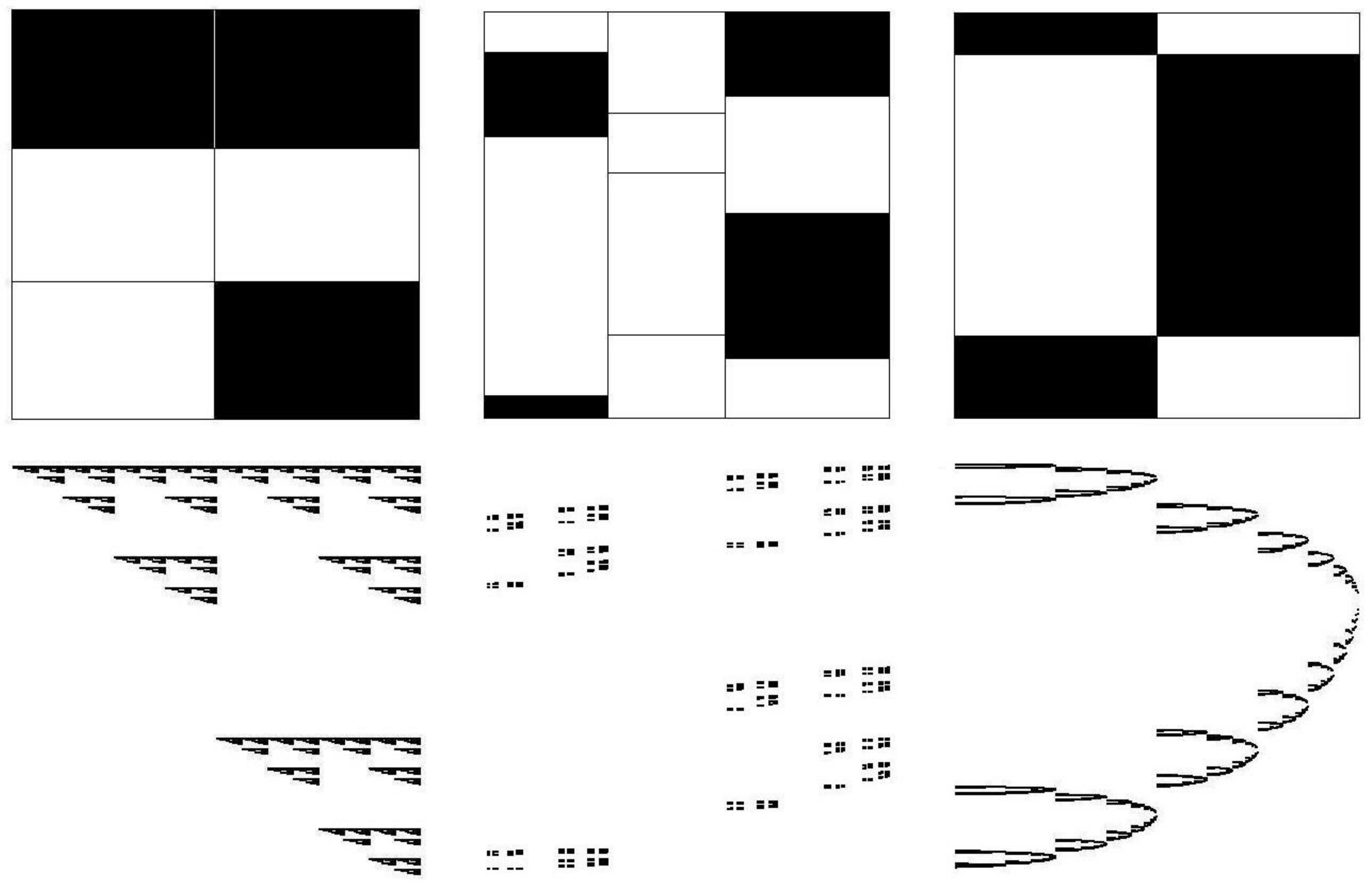}
\caption{The defining patterns for a Bedford-McMullen carpet, a Lalley-Gatzouras carpet and a Bara\'nski carpet from left to right and the corresponding attractors.}
\end{figure}

The first such class was introduced by Bedford \cite{bedford} and McMullen \cite{mcmullen} in the mid 1980s and is defined as follows. Take the unit square and divide it up into an $m \times n$ grid for some $m,n \in \mathbb{N}$ with $1 < m < n$.  Then select a subset of the rectangles formed by the grid and consider the IFS consisting of the affine maps which map the square onto each chosen rectangle, preserving orientation. The \textit{Bedford-McMullen carpet} is the self-affine attractor of the resulting IFS. Lalley and Gatzouras \cite{lalley-gatz} generalised this construction by allowing the columns to have varying widths and be divided up, independently, with the only restriction being that the height of each chosen rectangle had to be strictly smaller than the length of the base.  Another, more recently introduced, class was considered by Bara\'nski \cite{baranski}.  This time divide the unit square up into an arbitrary mesh of rectangles by slicing horizontally and vertically a finite number of times, of course at least once in each direction, choose a collection of subrectangles, and define an IFS as before. 

Even more general classes of self-affine carpet have been introduced and studied by Feng-Wang \cite{fengaffine} and Fraser \cite{fraseraffine}, but we omit detailed discussions of these because our results do not a priori apply in this setting.  See Section \ref{sec:prospects} for more details.

Although self-affine carpets are subsets of $[0,1]^2$, when viewing them from a dynamical point of view it is often more convenient to consider them as subsets of the torus $\mathbb{T}^2$. In particular, Bedford-McMullen carpets viewed in this way are $T_{m,n}$ invariant.

We also note that as $\mu$ is a Bernoulli measure for $T_{m,n}$, it is a Bernoulli measure supported on a Bedford-McMullen carpet (which may be the entire square). Set theoretical scaling sceneries of Bedford-McMullen and Lalley-Gatzouras carpets were studied by K\"aenm\"aki and Bandt in \cite{bandtkaenmaki}, where they obtained that under minor assumptions including that there be at least one box chosen in every column, the `tangent sets' of these carpets are of the form $[0,1] \times C$ where $C$ is a suitable Cantor set. Our work is certainly inspired by their work and the reason to go beyond sets to measures is to have the machinery of CP-chains at our disposal.

Using the fact that $T_{m,n}$-Bernoulli measures generate ergodic CP-chains, we can use the technology provided by Furstenberg, Hochman and Shmerkin to obtain several geometric applications, which we will now detail.

\subsection{Projection theorem}

Our first application of Theorem \ref{thm:main} is a projection theorem for Bernoulli measures for $T_{m,n}$ that is analogous to Theorem \ref{hocshmproj}.

\begin{thm}
\label{thm:proj}
If $\log m / \log n$ is irrational, then for a Bernoulli measure $\mu$ for $T_{m,n}$ we have
$$\dim \pi\mu = \min\{1,\dim \mu\}$$
for every $\pi \in \Pi_{2,1} \setminus\{\pi_1,\pi_2\}$.
\end{thm}

The above result follows from a more general phenomenon, first observed in \cite[Section 10]{HocShm12}.  Namely, if a measure $\mu$ generates a CP-chain through a partition operator which is driven by an irrational rotation then the above statement holds.  Similar irrationality conditions have appeared in other works searching for the exact set of projections that preserve dimension, see for example \cite{mor98, fjs, nps, ps}.

The following projection theorem by Ferguson, Jordan and Shmerkin in \cite{fjs} is particularly relevant to us.

\begin{thm}[Theorem 1.1 of \cite{fjs}]
\label{thm:fergjordshm}
If $K$ is a self-affine carpet in the Bedford-McMullen, Lalley-Gatzouras or Bara\'nski class which is of irrational type, then
$$\dim \pi(K) = \min\{1,\dim K\}$$
for every $\pi \in \Pi_{2,1} \setminus\{\pi_1,\pi_2\}$.
\end{thm}

Here the \textit{irrational type} condition for $m \times n$ Bedford-McMullen carpets simply means that $\log m / \log n$ is irrational. Theorem \ref{thm:proj} could be considered as a measure theoretical extension of Theorem \ref{thm:fergjordshm}.

\subsection{Distance sets} Our second application of Theorem \ref{thm:main} concerns the dimension of distance sets.  Given an analytic set $K \subset \mathbb{R}^d$, the \textit{distance set} $D(K)$ of $K$ is defined as the set of all distances in $K$, in other words
\[
D(K) = \{ \lvert x - y \rvert : x,y \in K\}.
\]
Falconer's distance set problem, originating with the paper \cite{falc}, concerns different variants on relating the size of $D(K)$ to the size of $K$.  One version of this is the following well-studied problem on dimensions:
\begin{conj}[Distance set conjecture]
Let $d \geq 2$ and let $K \subset \mathbb{R}^d$ be analytic. If $\dim K \geq d/2$, then 
$$\dim D(K)  = 1.$$
Furthermore, if $\dim K > d/2$, then $D(K)$ has positive Lebesgue measure.
\end{conj} 

This conjecture has gained a lot of attention in recent years due to its links with harmonic analysis and additive combinatorics, see for example the papers Hofmann and Iosevich \cite{hofm} and Katz and Tao \cite{katz} for discussions. The current state of the conjecture is that if one assumes 
$$\dim K > \frac{d}{2}+\frac{1}{3},$$
then $D(K)$ has positive Lebesgue measure. In the plane this was proved by Wolff \cite{wolf} in 1999, and for general $d$ by Erdogan \cite{erdogan2006} in 2006 by adapting Tao's bilinear Fourier restriction estimates. For the dimension part, Falconer \cite{falc} proved that for a planar set $K$ with $\dim K \geq 1$ we have $\dim D(K) \geq 1/2$. Bourgain \cite{bourgain2003} then extended this result by showing that there is a constant $c > 1/2$ such that if a set in the plane has $\dim K > 1$, then 
$$\dim D(K) > c.$$ 
Unfortunately, this constant $c$ is far from $1$ so the conjecture remains open, even in the plane. 

However, if $K$ supports a measure of full dimension such that the scaling scenery satisfies suitable statistical regularity, then we can say more. Our main result on distance sets, which builds on \cite{HocShm12}, gives a criterion for the support of a measure to satisfy the dimension part of the distance set conjecture; see Section \ref{sec:cpchain} for definitions of terminology used here. 

\begin{thm} \label{thm:distset}
Let $\mu$ be a measure on $\R^2$ which generates an ergodic CP-chain and is supported on a set $\spt \mu$ with positive length. Then
\[
\dim D(\spt \mu) \geq \min\{ 1, \dim \mu\}.
\]
\end{thm}

Of course, general measures need not generate ergodic CP-chains and so our theorem cannot be used to deal with the general case.  However, measures satisfying the hypothesis of Theorem \ref{thm:distset} in fact arise as \textit{dyadic micromeasures} of \textit{any} measure in $\R^2$. In other words, as accumulation measures of blow-ups using dyadic cubes, see Definition \ref{CPmicro}. This is once again a testament of the phenomenon that tangents may enjoy more regularity. 

\begin{cor}
\label{dyadicmicro}
Suppose $\mu$ is any measure on $\R^2$ with $\dim \mu > 1$. Then after randomly translating $\mu$, there are dyadic micromeasures $\nu$ of $\mu$ at $\mu$ almost every point with
$$\dim D(\spt \nu) = 1.$$
\end{cor}

Other examples which satisfy the conditions of Theorem \ref{thm:distset} are measures with \textit{uniformly scaling scenery} in the sense of Gavish \cite{gavish}. For more concrete applications, one can look to dynamically defined sets. Combining Theorem \ref{thm:distset} with our main result on the generation of a CP-chain for $T_{m,n}$ Bernoulli measures, we obtain the following corollaries.

\begin{cor} \label{cor:selfaffinedist}
If $K$ is a self-affine carpet in the Bedford-McMullen class with $\dim K  \geq 1$, then
$$\dim D(K) = 1.$$ 
Furthermore, this also holds if $K$ is in the Lalley-Gatzouras or Bara\'nski class with $\dim K  >1$.
\end{cor}

Finally, Theorem \ref{thm:distset} also allows us to recover the main result of 
Orponen \cite{orponen} directly from the construction of a CP-chain.

\begin{cor} \label{cor:selfsimdist}
If $K$ is a self-similar set in the plane with $\mathcal{H}^1(K)>0$, then
$$\dim D(K) = 1.$$
\end{cor}

Corollary \ref{cor:selfsimdist} was first obtained in \cite{orponen}, where the proof was divided into two cases.  The case where at least one map contained an irrational rotation was dealt with directly using the results of Hochman and Shmerkin on $C^1$ images of self-similar sets and the other case was dealt with via a careful geometric argument and without any need for the machinery of CP-chains. In higher dimensions, Falconer and Jin \cite{jin} presented a proof for self-similar sets with the extra assumption that the group generated by the rotational components of the maps is dense in the full orthogonal group - the natural higher dimensional analogue of the assumption that one map contains an irrational rotation.

The assumption $\cH^1(\spt \mu)>0$ in Theorem \ref{thm:distset} seems somewhat arbitrary and on closer inspection of the proof, it is evident that all we really need is the existence of points $x,y \in \spt \mu$ such that projection onto the line spanned by $x-y$ is one of the `good projections' for the CP-chain. In spirit the `good projections' are related to the `good projections' in Marstrand's projection theorem, i.e. those $\pi \in \Pi_{2,1}$ when
$$\dim \pi \mu = \min\{1,\dim \mu\}.$$
Taking this into account we obtain a technical result concerning when the set of `bad directions' is countable. In order to avoid introducing technical notation on CP-chains here, we defer the statement of this result until Section \ref{technicalprop} where it will be given as Proposition \ref{thm:distset2}. 

This technical proposition gives us information about the dimension of $D(K)$ when $\dim K \leq 1$. A natural conjecture in this setting would be that $\dim D(K) \geq \dim K$.  The $T_{m,n}$ Bernoulli measures considered in this paper already gives us examples of measures satisfying the hypotheses of Proposition \ref{thm:distset2} due to our projection result Theorem \ref{thm:proj}. As mentioned before in Theorem \ref{thm:fergjordshm}, for self-affine carpets that are of irrational type in the sense of Ferguson-Jordan-Shmerkin \cite{fjs}, the exceptional set for Marstrand's projection theorem has at most two directions. Thus we obtain the following general statement.

\begin{cor} \label{thm:distsetirrational}
If $K$ is a self-affine carpet in the Bedford-McMullen, Lalley-Gatzouras or Bara\'nski class which is of irrational type, then
$$\dim D(K) \geq \min\{1, \dim K\}.$$
\end{cor}

Interestingly, in light of Proposition \ref{thm:distset2} below, understanding the exceptional set in Marstrand's projection theorem is related to the distance set conjecture. For example, Corollary \ref{thm:distsetirrational} leaves open the question on rational type self-affine carpets with small dimension. As far as we know, in this case the size of the exceptional set of Marstrand is not fully understood.

We believe that Theorem \ref{thm:distset} could be applied to other dynamically defined sets and measures which enjoy such uniformly scaling dynamics.  Found, for example, when studying more general self-affine sets than carpets or general $T_{m,n}$ invariant sets. 

Unfortunately, in general the machinery of Hochman and Shmerkin is tailored only to tackle problems regarding dimension. Thus in order to get information about the Lebesgue measure of the distance set using CP-chains, one would need to study absolute continuity analogues of the results in \cite{HocShm12}.

\section{Magnification dynamics}
\label{sec:cpchain}

\subsection{CP-chains} We will first set up some terminology for CP-chains introduced and used by Hochman and Shmerkin in \cite{HocShm12} before we present our construction and proof. In this paper, a \textit{measure} is always a Borel probability measure on some metric space $K$. Write $\cP(K)$ as the space of measures supported on $K$ and $\spt \mu \subseteq K$ for the support of a measure in $\cP(K)$. If $K$ is compact, then $\cP(K)$ is a compact metric space with the \textit{Prokhorov distance}, defined by
$$d(\mu,\nu) = \inf\{\epsilon > 0 : \mu(A) \leq \nu(A_\epsilon)+\epsilon \text{ and } \nu(A) \leq \mu(A_\epsilon)+\epsilon \text{ for all Borel sets }A\},$$
where $A_\epsilon$ is the open $\epsilon$-neighborhood of $A$ in $K$. This metric induces the weak topology on $\cP(K)$, see for example Billingsley \cite[Appendix III, Theorem 5]{Billingsley1968}. Moreover, if $f : K \to K'$ is a map and $\mu \in \cP(K)$, we define $f\mu \in \cP(K')$ as the push-forward measure $\mu \circ f^{-1}$. A \textit{distribution} is a measure on $\cP(K)$ for some metric space $K$. If we write $x \sim \mu$ for a measure $\mu$, we mean that $x$ is a random variable whose distribution is given by $\mu$.

\begin{definition}[Boxes and blow-ups] \label{def:box} Let $B \subset \R^d$ be a \textit{box}, that is, a product of intervals in $\R$ which can be open, closed or half-open. Let $T_B : \R^d \to \R^d$ be the orientation preserving affine map defined by
$$T_B(x) = \frac{1}{|B|^{1/d}}(x-\min \overline{B}),$$
where $|B|$ is the volume of $B$ and $\min \overline{B}$ is the minimal element of $\overline{B}$ with respect to the lexicographical order. Define the \textit{normalised box} $B^* := T_B(B)$. Note that now the volume $|B^*| = 1$. If $\mu$ is a Borel probability measure on $\mathbb{R}^d$ and $B$ is a box with $\mu(B) > 0$, write
$$\mu^B = \frac{1}{\mu(B)} T_B(\mu|_B),$$
which is a probability measure supported on $B^*$.
\end{definition}

\begin{defn}[Partition operator and filtrations] Let $\cE$ be a collection of boxes in $\R^d$. A \textit{partition operator} $\Delta$ on $\cE$ attaches to each $B \in \cE$ a partition $\Delta B \subset \cE$ such that if $S$ is a homothety of $\R^d$, then $S(\Delta B) = \Delta(SB)$ for $B,SB \in \cE$. Given $B \in \cE$, a partition operator $\Delta$ defines a filtration of $B$ by
$$\Delta^0(B) = \{B\} \quad\text{and}\quad \Delta^{n+1}(B) = \bigcup_{A \in \Delta^n(B)} \Delta(A).$$
A partition operator $\Delta$ is $\delta$-\textit{regular} if for any $B \in \cE$ there exists a constant $c > 1$ such that for any $n \in \N$ any element $A \in \Delta^n(B)$ contains a ball of radius $\delta^n/c$ and is contained in a ball of radius $c\delta^n$.
\end{defn}

\begin{definition}[CP-chain]\label{cpchain} Fix a collection of boxes $\cE$ and define a state space
$$\Omega = \{(B,\mu) : \mu \in \cP(B^*), B \in \cE\}.$$
A \textit{$\delta$-regular CP-chain} $Q$ with respect to a $\delta$-regular partition operator $\Delta$ is a stationary Markov process on the state space $\Omega$ with the Markov kernel
$$F(B,\mu) = \sum_{A \in \Delta(B^*)} \mu(A) \delta_{(A,\mu^B)}, \quad (B,\mu) \in \Omega,$$
which we call the \textit{Furstenberg kernel}. In other words, given $A \in \Delta(B^*)$, then the process moves from $(B,\mu)$ to $(A,\mu^B)$ with probability $\mu(A)$.
\end{definition}

When we say `let $Q$ be a CP-chain', we also mean by $Q$ the unique stationary measure with respect to the chain. Then by definition $Q$ is a shift invariant measure for the dynamical system $(\Omega^\N,\sigma)$. This way we can also define the CP-chain to be \textit{ergodic} if the measure preserving dynamical system $(\Omega^\N,\sigma,Q)$ is ergodic in the usual sense: any $\sigma$-invariant set has either measure $0$ or $1$. Thus for an ergodic CP-chain, we have the ergodic theorem at our disposal: if $f : \Omega^\N \to \R$ is continuous, then for a $Q$ typical realisation $X_n = (B_n,\mu_n)$, $n \in \N$, we have
$$\frac{1}{N} \sum_{k = 0}^{N-1} f(X_{k+1},X_{k+2},\dots) \to \int f \, dQ, \quad \text{as } N \to \infty.$$


\begin{definition}[Generating CP-chains and magnification dynamics]\label{CPmicro} Let $\tilde Q$ be the measure component of a CP-chain $Q$ with a partition operator $\Delta$ on a collection of boxes $\cE$. Given $B \in \cE$, the CP-chain $Q$ is \textit{generated} by a measure $\mu \in \cP(B^*)$ if at $\mu$ almost every $x \in B^*$ the \textit{CP scenery distributions}
$$\frac{1}{N} \sum_{k = 0}^{N-1} \delta_{\mu^{\Delta^k(x)}}$$
converge weakly to $\tilde Q$ as $N \to \infty$, and for any $q \in \N$ the $q$-\textit{sparse} CP scenery distributions
$$\frac{1}{N} \sum_{k = 0}^{N-1} \delta_{\mu^{\Delta^{qk}(x)}}$$
converge to some, possibly different, distributions $\tilde Q_q$ as $N \to \infty$. The convergence of the $q$-sparse CP scenery distributions is important when transferring local geometric information back to $\mu$.  In particular, note the appearance of $q$ in the following subsection. It is a consequence of the ergodic theorem that $\tilde Q$ almost every measure $\nu$ in fact generates $Q$, see \cite[Proposition 7.7]{HocShm12}.

The measures $\mu^{\Delta^k(x)}$ appearing in the CP scenery distributions are called \textit{minimeasures} $\mini_\Delta(\mu,x)$ associated to $\Delta$ of $\mu$ at $x$ and accumulation points of them in the space of probability measures are called \textit{micromeasures} $\micro_\Delta(\mu,x)$ associated to $\Delta$ of $\mu$ at $x$. If $\Delta$ is the partition operator associated to the dyadic cube decomposition of $[0,1]^d$, then we call these \textit{dyadic} mini- and micromeasures.
\end{definition}

\subsection{Dimensions and projections of CP-chains}

For our applications, we extensively use the results of \cite{HocShm12} and especially their projection theorems.  In this subsection we briefly recall some of the key definitions and results. Recall that a measure is \textit{exact dimensional} if the local dimension exists and is equal to a constant at almost every point.  In this case, the almost sure common value for the local dimension must be $\dim \mu$.

\begin{definition}[Dimension of a CP-chain] The \textit{dimension} of a CP-chain is the average Hausdorff dimension $\dim \nu$ of measures according to the measure component $\tilde Q$.  Namely,
$$\dim Q = \int \dim \nu \, d\tilde Q(\nu).$$
Recall that the map $\mu \mapsto \dim \mu$ is a Borel function with respect to the weak topology so we are allowed to define the integral here. It is a consequence of the ergodic theorem that if $Q$ is ergodic, then $\tilde Q$ almost every $\nu$ has exact dimension $\dim Q$, see \cite[Lemma 7.9]{HocShm12}. Moreover, for example from the proof of \cite[Lemma 7.9]{HocShm12}, one can see that if $\mu$ generates a CP-chain $Q$, then $\mu$ is exact dimensional with $\dim \mu = \dim Q$. 
\end{definition}

The main quantitative projection theorem of \cite{HocShm12} presented links between the expected dimension of projections of micromeasures and the dimensions of the projections of measures that generate $Q$. 

\begin{notation}[Expected dimension of projections] \label{mapE} Let $Q$ be a $\delta$-regular CP-chain. Fix $r > 0$ and a measure $\mu$. The \textit{$r$-entropy} of $\mu$ is defined by
$$H_{r}(\mu) = - \int \log \mu(B(x,r)) \, d\mu (x).$$
Then the fraction $H_r(\mu)/\log (1/r)$ can be thought as a kind of `$r$-discrete dimension' of $\mu$ at the scale $r > 0$. Given $\pi \in \Pi_{d,k}$, that is, an orthogonal projection $\pi : \R^d \to \R^k$, we write
$$E_q(\pi) = \int \frac{1}{q\log(1/\delta)} H_{\delta^q} (\pi\nu)\, d\tilde Q(\nu),$$
that is, the expected $\delta^q$-discrete dimension of $\pi \nu$ when $\nu$ is drawn according to $\tilde Q$.
\end{notation}

\begin{prop}[Theorem 8.2 of \cite{HocShm12}] \label{semicproj}
Let $Q$ be an ergodic CP-chain. Then the limit $E = \lim_{q \to \infty} E_q$ exists and $E : \Pi_{d,k} \to \R$ is lower semicontinuous with the properties:
\begin{itemize}
\item[(1)] for any $\pi \in \Pi_{d,k}$ we have at $\tilde Q$ almost every $\nu$ that
$$ \dim \pi \nu =  E(\pi) $$
\item[(2)] for almost every $\pi \in \Pi_{d,k}$ we have
$$E(\pi) = \min\{k,\dim Q\}$$
\item[(3)] for any measure $\mu$ that generates $Q$ and for any $\pi \in \Pi_{d,k}$ we have
$$\dim \pi \mu \geq E(\pi)$$
\end{itemize}
\end{prop}

This result is useful for proving results on projections and distance sets in our later applications. The advantage of the nice formula for the map $E$ is that it allows us to obtain the following non-linear projection theorem as well, which will be crucial in the application to distance sets.

\begin{prop}[Proposition 8.4 of \cite{HocShm12}] \label{prop8.4}
Let $\pi \in \Pi_{d,k}$. Suppose that a measure $\mu$ on $[0,1]^d$ generates an ergodic $\delta$-regular CP-chain $Q$. Then for all $C^1$ maps $g : [0,1]^d \to \R^k$ such that the maximal norm
$$\sup_{x \in \spt \mu} \|D_x g - \pi\| < \delta^q,$$
we have
$$\dim g \mu \geq E_q(\pi) - C/q,$$
where $C$ depends only on $\delta$, $d$ and $k$.
\end{prop}

The number $q$ present in the above propositions comes from the definition of generating distributions, when we require that not only the CP scenery distributions converge, but also the $q$-sparse distributions. 

\section{Encoding the torus and conformal re-scaling}

\subsection{Shift spaces and Markov partitions} \label{markov} 

First we set up some standard notation. Let
$$I=\{0,1,\dots,m-1\} \quad \text{and} \quad J=\{0,1,\dots,n-1\}$$ 
be the alphabets corresponding to the horizontal and vertical directions respectively, where we assume $m < n$. Let $I^* = \bigcup_{k \in \mathbb{N}} I^k$ and $J^* = \bigcup_{k \in \mathbb{N}} J^k$ be the sets of finite words and $I^\infty$ and $J^\infty$ be the sets of infinite words over $I$ and $J$ respectively. We will denote elements $\i \in I^* \cup I^\infty$ by $\i = i_0 i_1 \dots $ and $\j \in J^* \cup J^\infty$ by $\j = j_0 j_1 \dots$. Let $\Sigma = I^\infty \times J^\infty$ be the symbolic space of pairs. In $I^\infty$ the distance between two words $\i$ and $\i'$ is given by $m^{-|\i \wedge \i'|}$ and between words $\j$ and $\j'$ in $J^\infty$ by $n^{-|\j \wedge \j'|}$ where $\i \wedge \i' \in I^* \cup \{\omega\}$ is the common part of the words $\i$ and $\i'$ and where $\omega$ is the empty word. In the product space $\Sigma$, we use the maximum metric obtained from $I^\infty$ and $J^\infty$. In a slight abuse of notation we let $\sigma$ denote the left shifts on both of the spaces $I^\infty$ and $J^\infty$.  For $\i \in I^*$ we define the cylinder $[\i] \subseteq I^\infty$ as
\[
[\i] = \{ \i \i' : \i' \in I^\infty\},
\]
with cylinders in $J^\infty$ defined similarly. Finally, by slightly abusing the notation, let $\pi_1 : \Sigma \to I^\infty$ and $\pi_2 : \Sigma \to J^\infty$ be the coordinate projections.

The natural Markov partition $\xi$ related to the maps $T_{m,n}$ on the torus can be given by decomposing $\T^2$ to $mn$ rectangles of width $1/m$ and height $1/n$ orientated with the  coordinate axes. These elements are closed but their interiors are disjoint. The generation $k$ refinement of the Markov partition $\xi$ is then defined by the join
$$\xi_k = \bigvee_{i=0}^{k-1} T^{-i}_{m,n} \xi$$
so that $\xi_1 = \xi$, see for example \cite{walters} for the definition. Now each element $R \in \xi_k$ is a rectangle of width $1/m^k$ and height $1/n^k$, and by enumerating each element of $\xi$ with indices $i \in I$ and $j \in I$, we can encode $R = R(\i,\j)$ for some finite sequences $\i \in I^k$ and $\j \in J^k$. Since the diameters of each element in $\xi_k$ is decreasing to $0$ as $k \to \infty$, we have that $(\i,\j) \in \Sigma$ corresponds to a unique point $\xi(\i,\j) \in \T^2$, which defines a \textit{coding map} or \textit{symbolic projection} $\xi : \Sigma \to \T^2$ mapping $\xi(\Sigma) = \T^2$.

\begin{remark} \label{adicproj} Currently, the notion of CP chain is only defined in the Euclidean geometry of $\R^d$. However, in our case the measures $\mu$ we study in the torus can be also identified as measures supported on the square $[0,1]^2$ as they are defined symbolically. By abusing the notation, we can define another map $\xi : \Sigma \to [0,1]^2$ by considering the $m$-adic and $n$-adic expansions the sequences $(\i,\j) \in \Sigma$ generate:
$$\xi(\i,\j) := \sum_{k = 0}^\infty (i_km^{-k},j_kn^{-k}).$$
This map is then a continuous transformation from $\Sigma \to [0,1]^2$ with respect to the Euclidean topology. Moreover, then we can similarly define the elements $R(\i',\j') \subset [0,1]^2$ for $\i' \in I^*$ and $\j \in J^*$ as push-forwards
$$R(\i',\j') := \xi([\i'] \times [\j']).$$
All our measures we consider on $\T^2$ (in the next section) are defined symbolically in $\Sigma$, and then projected down to the torus. Thus when we talk about generation of CP chains of measures defined on the torus, we will consider the appropriate measure defined on $[0,1]^2$ (defined via the projection $\xi : \Sigma \to [0,1]^2$).
\end{remark}

\subsection{Bernoulli measures and disintegration} \label{bernoulli}

Having finally set up a symbolic space and a Markov partition to encode the dynamics of $T_{m,n}$, we can define Bernoulli measures for that action.

\begin{definition}[Bernoulli measures]\label{bernoullimeasures} A measure $\mu$ on the symbolic space $\Sigma$ is \textit{Bernoulli} if there are weights $p_{ij} \in [0,1]$, $i \in I$ and $j \in J$ such that
$$\mu([\i i] \times [\j j]) = p_{ij}\mu([\i] \times [\j ]), \quad i \in I, j \in J,$$
for any given words $\i \in I^*$ and $\j \in J^*$. A measure $\mu'$ on $\T^2$ (or $\R^2$) is \textit{Bernoulli for} $T_{m,n}$ if $\mu' = \xi \mu$ for some Bernoulli measure $\mu$ on $\Sigma$. Note that we will often use the same notation for $\mu$ and $\mu'$ in the proceedings.
\end{definition}

To take into account the faster expanding direction in the symbolic space, we create a disintegration of the Bernoulli measure $\mu$ along the fast expanding direction.  For a Bernoulli measure this has the following explicit representation. 

\begin{notation}[Disintegration and conditional measures $\mu_\i$]
 For $i \in I$ write
$$q_i = \sum_{j \in J} p_{ij} \quad \text{and} \quad p_i(j) = \begin{cases}
p_{ij}/q_i, & \text{if } q_i>0\\
0, & \text{if } q_i=0
\end{cases}, \quad j \in J.$$ 
The conditional measure $\mu_\i \in \cP(J^\infty)$ of $\mu$ at $\i \in I^\infty$ is determined by
$$\mu_\i([\j]) := p_{i_0}(j_0)\cdots p_{i_{k-1}}(j_{k-1}), \quad \j \in J^k \text{ and } k \in \N,$$
where $[\j] \subset J^\infty$ is the symbolic cylinder corresponding to $\j$.  Let $\pi_1$ be the projection of $\Sigma$ onto the first coordinate. Given a Bernoulli measure on $\Sigma$, the measure $\pi_1 \mu$ is the Bernoulli measure on $I^\infty$ with weights $q_0,q_1,\dots,q_{m-1}$ and we can recover $\mu$ by integrating over the conditional measure $\mu_\i$, i.e.
$$\mu(E) = \int  \mu_\i(E_\i) \, d\pi_1 \mu(\i)$$
for Borel $E  \subset \Sigma$ and where $E_\i = \{ \j \in J^\infty : (\i , \j) \in E\}$.
\end{notation}

\subsection{Approximate squares and the associated CP-chains} \label{approximatesquares}

Now we present the construction of the CP-chain we are going to use. The first step is to construct a regular filtration of $[0,1]^2$ using rectangles with bounded eccentricities. The basic sets used in this filtration are widely known as \textit{approximate squares}. We follow the same definition of a partition operator as Hochman and Shmerkin use in \cite[Section 10.2]{HocShm12} when dealing with product measures.

\begin{defn}[Partition operator given by the approximate squares] Given $t \in [0,1]$ we write
$$R_t = [0,1] \times [0,n^t].$$
Let $\cE$ be the collection of all boxes in $\R^2$ that can be obtained from some $R_t$, $t \in [0,1)$, using a re-scaling and a translation. We define a partition operator $\Delta$ first on the sets $R_t$, which then defines $\Delta$ also on the whole of $\cE$ if we extend it by a similitude. Define the angle
$$\alpha := \frac{\log m}{\log n}.$$
Depending on the relationship between $t$ and $\alpha$ we obtain different partitions.
\begin{itemize}
\item[(1)] First we split $R_t$ into $m$ rectangles $R_t(i)$, $i \in I$, of width $1/m$ and height $n^t$. 
\item[(2)] If $t \geq 1-\alpha$, then further partition $R_t(i)$ into $n$ rectangles $R_t(i,j)$, $j \in J$, of height $n^{t-1}$.
\end{itemize}
Here we make the choices of rectangles to be so that the obtained rectangles are `upper-right half-open' that is, of the form $[a,b) \times [c,d)$, except that when we touch the boundary of $R_t$, we include the boundary $[a,b] \times [c,d]$. Thus $\Delta(R)$ is either a partition of $R$ by $m$ or $mn$ rectangles that are similar to $R_{\phi(t)}$, where $\phi : \T \to \T$ is the rotation
$$\phi(t) = t+\alpha \mod 1, \quad t \in \T.$$
This way we obtain an $\tfrac{1}{m}$-regular partition operator $\Delta$. Notice that all the rectangles $R_t \subset R_1$ for $t \in [0,1)$.
\end{defn}

We deviate slightly from the notation used in \cite{HocShm12} on the rotation part. There a $\log m$ rotation was used on $[0,\log n)$, but to slightly lighten the notation in the proofs, we have renormalised this to the $\alpha$-rotation on $[0,1)$.

The rotation $\phi$ of angle $\alpha$ will control the eccentricities of the rectangles when acting with $\Delta$. For $k \in \N$ and $t \in \T$, let $\ell_t(k)$ be the number of times we have revolved around $\T$ during the first $k$ steps. In other words,
$$\ell_t(k) = \sharp \{0 \leq k' \leq k-1 : \phi^{k'}(t) \geq 1-\alpha\}.$$
Thus $\ell_t(k)$ is the integer part of $\alpha k + t$ and so $k - \ell_t(k) \to \infty$ as $k \to \infty$ since $m < n$. Then by the definition of $\Delta$ and the rectangles induced by the Markov partition, we have
$$\overline{\Delta^k(\xi(\i,\j))} = R(\i|_k,\j|_{\ell_0(k)}).$$
Having fixed the numbers $m < n$, let $S_t : \R^2 \to \R^2$ be the hyperbolic matrix
$$S_t = \begin{pmatrix}n^{-t/2} & 0 \\ 0 & n^{t/2}\end{pmatrix}$$
which maps $[0,1]^2$ onto the normalised box $R_t^*$, recall Definition \ref{def:box}.

\begin{definition}[CP-chains $Q$ and $Q_\alpha$] \label{ourchain} Let $\mu$ be a Bernoulli measure for $T_{m,n}$. Depending on whether $\alpha$ is irrational or not, $\mu$ will generate different CP-chains.

\begin{itemize}
\item[(1)] In the irrational case, we obtain a CP-chain $Q$, whose measure component is defined to be the distribution of the random measure
$$S_t(\pi_1 \mu \times \mu_x),$$
when $x \sim \pi_1 \mu$ and $t \sim \lambda$, where $\mu_x$ is the conditional measure of $\mu$ with respect to the vertical fiber at $x \in [0,1]$ and $\lambda$ is the Lebesgue measure on $[0,1]$. 

\item[(2)] When $\alpha = p/q \in \Q$ with $gcd(p,q) = 1$, we obtain a CP-chain $Q_\alpha$, whose measure component is defined to be the distribution of the random measure
$$S_t(\pi_1 \mu \times \mu_x),$$
when $x \sim \pi_1 \mu$ and $t \sim \tau_\alpha$, where $\mu_x$ is the conditional measure of $\mu$ with respect to the vertical fiber at $x \in [0,1]$ and
$$\tau_\alpha := \frac{1}{q}\sum_{k = 0}^{q-1} \delta_{\phi^k(0)}.$$ 
\end{itemize}

Now the actual CP-chain is defined by choosing first the measure $\nu$ according to the measure component defined above, and then for $B \in \cE$ and $A \in \Delta(B)$, the transition probability to move from $(B,\nu)$ to $(A,\nu^A)$ is set to be $\nu(A)$.
\end{definition}

\section{Proof of the main result}
\label{sec:symbolic}
In this section we prove Theorem \ref{thm:main}. There are several ingredients required and so we first explain the outline of the proof. The main bulk of the proof is certainly in proving that $\mu$ generates $Q$ and for that we proceed with the following steps in mind:

\begin{itemize}
\item[(1)] Reformulate the desired result in terms of a symbolic skew product system $(\T \times \Xi,Z)$ driven by a rotation, where the rotational part takes care of the change in eccentricities of the approximate squares.  This will culminate in the statement of Proposition \ref{thm:generatedirrational} and a discussion of why proving this gives Theorem \ref{thm:main}.
\item[(2)] Introduce a \textit{parametrisation} of micro- and minimeasures of $\mu$ defined by the slower expanding direction $I^\infty$, which relies heavily on the Bernoulli properties of the measure. This allows us to reduce the magnification dynamics in the space of measures to a question about shift spaces.
\item[(3)] To obtain the symbolic reformulation Proposition \ref{thm:generatedirrational} we first argue that it suffices to use a class of simple test functions, rather than all continuous functions, which is possible by the special choices of the simple functions.
\item[(4)] The advantage of using simple functions in this context is that they take into account the symbolic parametrisation of the mini- and micromeasures. Then we just prove an equivalent statement in shift spaces, which is a law of large numbers for weakly correlated random variables.
\end{itemize}

\subsection{Symbolic reformulation}
\label{subsection1}

First we identify the transition dynamics induced by the approximate square partition operator $\Delta$ to introduce a dynamical system on the phase space 
$$\Xi = \{(\i,\j,\nu) \in \Sigma \times \cP(\Sigma) : (\i,\j) \in \spt \nu\}.$$ 

\begin{notation}[Symbolic magnification] The \textit{blow-up} $\nu^{C}$ of a measure $\nu$ in $\mathcal{P}(\Sigma)$ to a cylinder $C = [\i] \times [\j]$, $\i \in I^*$, $\j \in J^*$, with $\nu(C) > 0$, can be defined by restricting $\nu$ to $C$, shifting the restriction $\nu|_C$ back to the root of the tree and normalizing with $\nu(C)$. In other words,
$$\nu^C([\i'] \times [\j']) = \frac{1}{\nu(C)}\nu([\i\i'] \times [\j\j']) \quad \text{for $\i' \in I^*$ and $\j' \in J^*$.}$$ 
\end{notation}

Due to the non-conformality of the map $T_{n,m}$, in order to keep the eccentricity of the rescalings of the measure bounded, we are required to iterate the two directions at different rates.  

\begin{notation}(Magnification operator $Z$) For $t \in \T$ and $(\i,\j) \in \Sigma$ write
$$\sigma_t = \begin{cases}
\sigma, & \text{if } t \geq 1-\alpha\\
\id, & \text{otherwise.}
\end{cases} \quad \text{and} \quad C_t(\i,\j) = \begin{cases}
[i_0] \times [j_0], & \text{if } t \geq 1-\alpha\\
[i_0] \times J^\infty, & \text{otherwise.}
\end{cases}$$ 
Define a transformation $Z : \T \times \Xi \to \T \times \Xi$ as a skew product
$$Z(t,\i,\j,\nu) = (\phi(t),\sigma(\i),\sigma_t(\j),\nu^{C_t(\i,\j)}).$$
For $k \in \N$ define the iterated cylinder set
$$C_t^k(\i,\j) = [\i|_k] \times [\j|_{\ell_t(k)}].$$
If we move from the point $Z(t,\i,\j,\nu)$ to $Z^2(t,\i,\j,\nu)$ in the phase space, the measure coordinate $\nu^{C_t(\i,\j)}$ iterates to $\nu^{C_t^2(\i,\j)}$ since by checking the definition of the rescaled measure, we see that
$$(\nu^{C_t(\i,\j)})^{C_{\phi(t)}(\sigma(\i),\sigma_t(\j))} = \nu^{C_t^2(\i,\j)}.$$
Thus for $k \in \N$ and starting with the original Bernoulli measure $\mu$, the $k$th iterate is
$$Z^k(t,\i,\j,\mu) = (\phi^k(t),\sigma^k(\i),\sigma^{\ell_t(k)}(\j),\mu^{C_t^k(\i,\j)}).$$
The measures $\mu^{C_t^k(\i,\j)}$ arising from this process are called \textit{symbolic minimeasures} (\emph{of} $\mu$). Moreover, any accumulation point of the orbit $\{\mu^{C_t^k(\i,\j)}\}_{k \in \N}$ is called a \textit{symbolic micromeasure} (\emph{of} $\mu$).
\end{notation}

Given $t \in \T$ and $(\i,\j)$, the distribution of the orbit $\{Z^{k}(t,\i,\j,\mu)\}_{k \in \N}$ depends highly on the algebraic properties of the rotation $\phi$ we impose on the system. We obtain two different distributions depending on whether $\alpha = \tfrac{\log m}{\log n}$ is rational or irrational. If $\alpha$ is irrational, the $t$-coordinate equidistributes to Lebesgue measure $\lambda$ and when $\alpha$ is rational we equidistribute to a measure supported on a periodic orbit.

\begin{prop}
\label{thm:generatedirrational} Let $\mu$ be a Bernoulli measure on $\Sigma$. Let $P$ be the distribution of the triple $(\i,\j,\pi_1 \mu \times \mu_\i)$ where $\j \sim \mu_\i$ and $\i \sim \pi_1 \mu$. Suppose $\alpha$ is irrational. For any $q \in \N$ we have
\begin{equation}\label{eq:stronggen}\frac{1}{N} \sum_{k = 0}^{N-1} \delta_{Z^{qk}(t,\i,\j,\mu)} \to \lambda \times P \quad \text{as } N \to \infty\end{equation}
at every $t \in \T$ and $\mu$ almost every $(\i,\j)$.

If $\alpha = p'/q'$ is rational with $\gcd(p',q') = 1$, then \eqref{eq:stronggen} holds with $\lambda$ replaced with the measure
$$\tau_{t,q} = \frac{\gcd(q,q')}{q'} \sum_{k = 0}^{q'/\gcd(q,q') - 1} \delta_{\phi^{k\gcd(q,q') } (t)}$$
supported on $q'/\gcd(q,q')$ periodic orbit depending on $q$.

Moreover, the distributions $\lambda \times P$ and $\tau_{t,q} \times P$ are $Z$ invariant and $Z$ ergodic.
\end{prop}

\subsection{From symbolic to geometric model} \label{symbgeom}

Let us now show how Proposition \ref{thm:generatedirrational} gives Theorem \ref{thm:main}. The main idea is that the minimeasures $\mu^{\Delta^k(x,y)}$ are precisely the push forwards of the symbolic minimeasures $\mu^{C_0^k(\i,\j)}$ in $\cP(\Sigma)$ under the projection $\xi$. The reason that this is true is that we may assume that the measure of the boundary of any approximate square is zero, otherwise $\mu$ would be supported on a vertical or horizontal line.  This one to one correspondence passes to micromeasures, since micromeasures are just weak limits of minimeasures. Let us now make this more precise. 

First we need the following lemma on dealing with discontinuous test functions, which is proved for example in \cite[Theorem 2.7]{Billingsley1968}.

\begin{lemma} \label{weakbillingsley}
Let $X$ be compact metric space. If $\mu_N\to \mu$ weakly for probability measures $\mu_N$ and $\mu$ on $X$, and $f : X\to \R$ is $\mu$ almost everywhere continuous, then $\int f\,d\mu_N \to \int f \,d\mu$.
\end{lemma}

\begin{proof}[Proof of Theorem \ref{thm:main}] Define $E$ to be the union of all the boundaries of rectangles obtained in the Markov partition. In other words, we define the countable union
$$E = \bigcup_{\i \in I^*,\j \in J^*} \partial R(\i,\j),$$
recall the notation used in Section \ref{markov}. Let $\mu' = \xi \mu$ be a Bernoulli measure for $T_{m,n}$ (defined on $\R^2$, recall Remark \ref{adicproj} and Definition \ref{bernoullimeasures}) for some Bernoulli measure $\mu$ on $\Sigma$. We will assume from now on that $\mu'$ is not supported on any horizontal or vertical line.  If it was, then it would be a self-similar measure supported on an isometric copy of the unit interval which would simplify the subsequent arguments greatly, though the same results apply. A simple consequence of $\mu'$ not being contained in any horizontal or vertical line is that $\mu'(E) = 0$, i.e. the boundaries of all approximate squares carry zero weight.  From this it follows from the definition of the minimeasures that
\begin{align}\label{minisymb}\mu'^{\Delta^k(\xi(\i,\j))} = \mu'^{R(\i|_k,\j|_{\ell_0(k)})} = S_{\phi^k(0)}(\xi\mu^{C_0^k(\i,\j)})\end{align}
for any $k \in \N$ as $\Delta^k(\xi(\i,\j)) = R(\i|_k,\j|_{\ell_0(k)})$ except in at most a $\mu'$ null set. 

A similar property also holds for the measures in the support of the CP chain. First of all, the projection $\pi_1 \mu'$ is a Bernoulli measure on the unit interval with respect to the $m$-adic partition which does not contain any atoms (since we assumed $\mu'(E) = 0$). Moreover, as $\mu'(E) = 0$, the fibre $\mu'_{x}$ gives no mass to $n$-adic rational numbers for $\pi_1 \mu'$ typical $x$, so we must have $(\pi_1 \mu' \times \mu'_{x})(E) = 0$ for these $x$. Hence
\begin{align}\label{microsymb}\pi_1 \mu' \times \mu_{\xi_1(\i)}' = \xi(\pi_1 \mu \times \mu_\i),\end{align}
for $\pi_1 \mu$ typical $\i \in I^\infty$, where $\xi_1(\i) = \sum_{k = 0}^\infty i_km^{-k}$ is the point in $[0,1]$ with $m$-adic expansion given by $\i$.

Let us prove that along the orbit $(\mu'^{\Delta^k(\xi(\i,\j))})_{k \in \N}$ we generate the measure component  of the desired CP-chain. 
Define $\Theta : \T \times \Xi \to \cP(R_1)$ by
$$\Theta(t,\i,\j,\mu) = S_t(\xi \mu).$$
Notice that $\Theta$ is independent of the $\Sigma$ coordinate. However, $\Theta$ is not continuous in the $\T$-coordinate at $t = 0$, but this is the only way discontinuities can occur. Let us consider separately the case when $\alpha$ is irrational. Let $f : \cP(R_1) \to \R$ be a continuous test function and fix $q \in \N$. Then we have by the continuity of $\xi$ that
$$(\lambda \times P)((t,\i,\j,\nu) : f \circ \Theta \text{ is discontinuous at }(t,\i,\j,\nu)) = 0.$$
Thus by Proposition \ref{thm:generatedirrational} and Lemma \ref{weakbillingsley}, we have
\begin{align}\label{convergence}\frac{1}{N} \sum_{k = 0}^{N-1} (f\circ \Theta)(Z^{qk}(0,\i,\j,\mu)) \to \int f \circ \Theta \, d(\lambda \times P)\end{align}
at $\mu$ almost every $(\i,\j) \in \Sigma$. On the other hand, by \eqref{microsymb} the push-forward
$$\Theta(\lambda \times P) = \tilde Q$$
and also
$$\Theta(Z^{qk}(0,\i,\j,\mu)) = S_{\phi^{qk}(0)}(\xi \mu^{C_t^{qk}(\i,\j)}) = \mu'^{R(\i|_{qk},\j|_{\ell_0(qk)})} = \mu'^{\Delta^{qk}(\xi(\i,\j))}$$
so the proof follows as the full $\mu$ measure set where \eqref{convergence} holds gets mapped onto a full $\mu'$ measure set under $\xi$. This completes the proof of Theorem \ref{thm:main} in the irrational $\alpha$ case as the measure component $\tilde Q$ in both rational and irrational case is given by the push-forward $\Theta$ of an $Z$ ergodic and invariant distribution $\lambda \times P$, so we obtain stationarity and ergodicity for the chain $Q$ constructed.

The case when $\alpha$ is rational is similar, but slightly simpler, and so we omit the details.  The main difference is that the measures $\tau_{t,q}$ are supported on discrete sets (periodic orbits of $t$ under rotation by $q\alpha$) and so the test functions are automatically continuous so we do not need Lemma \ref{weakbillingsley}. 
\end{proof}

Thus it suffices to prove Proposition \ref{thm:generatedirrational} which will be the sole purpose of the next few sections.

\subsection{Parametrisation of mini- and micromeasures}
\label{subsection2}

Let us first introduce a parametrisation of the measures that arise or accumulate in the measure coordinate of the orbit $\{Z^k(t,\i,\j,\mu)\}_{k \in \N}$.

\begin{notation} Let $\Psi,\Psi_k : I^\infty \to \cP(\Sigma)$ be the maps defined by
$$\Psi(\i) =\pi_1 \mu \times \mu_\i \quad \text{and} \quad \Psi_k(\i)([\i'] \times [\j']) = \frac{1}{\pi_1 \mu[\i|_k]}\int_{[\i|_k \i']} \mu_\v[\j'] \,d\pi_1 \mu(\v)$$
for $\i \in I^\infty$ and $\i' \in I^*$ and $\j' \in J^*$.
\end{notation}

We notice that the image $\Psi_{k-\ell_t(k)}(I^\infty)$ is exactly the set of all $k$th generation minimeasures $\mu^{C_t^k(\i,\j)}$ of $\mu$ for large enough $k$. When blowing-up at $(\i,\j)$, provided we have iterated the blow-up enough times, we end up having exactly measures of the form $\Psi_{k-\ell_t(k)}(\i')$ for some other $\i' \in I^\infty$ depending on the history. The image $\Psi(I^\infty)$ is then a compact Cantor set in the space of measures and it consists of exactly the micromeasures of $\mu$ when evolving in $Z$. We will formalize all of this in the following lemma.

\begin{lemma} \label{miniandmicro}
 
\begin{itemize}\item[(1)] If $(t,\i,\j) \in \T \times \Sigma$ and $k > \ell_t(k)$, then the measure coordinate of $Z^k(t,\i,\j,\mu)$ is
$$\mu^{C_t^k(\i,\j)} = \Psi_{k-\ell_t(k)}(\sigma^{\ell_t(k)}(\i)).$$
\item[(2)] If $\i \in I^\infty$, then $\Psi_k(\i)$ and $\Psi(\i)$ agree on all cylinders of generation at most $k$. More precisely, we can say that if $\i'$ and $\j'$ are finite words in $I^*$ and $J^*$ respectively, then for any $k \geq |\j'|$ we have
$$\Psi_k(\i)([\i'] \times [\j']) = \Psi(\i)([\i'] \times [\j']).$$
In particular this shows that $\Psi_k(\i) \to \Psi(\i)$ as $k \to \infty$.
\end{itemize}
\end{lemma}
\begin{proof}
\begin{itemize}
\item[(1)] Write $\ell = \ell_t(k)$. Fix infinite words $\i \in I^\infty$ and $\j \in J^\infty$. Since $k > \ell$ and $\mu$ is a Bernoulli measure, we have
\begin{align*}\mu^{C_t^k(\i,\j)}([\i'] \times [\j']) &= \frac{\mu([\i|_k \i'] \times [\j|_{\ell}\j'])}{\mu([\i|_k] \times [\j|_{\ell}])} = \frac{\mu([\tilde \i \i'] \times [\j'])}{\mu([\tilde \i] \times J^\infty)}\\
& =\frac{1}{\pi_1\mu[\tilde \i]}  \int_{[\tilde \i \i']} \mu_\v[\j'] \, d\pi_1 \mu(\v) = \Psi_{k-\ell}(\sigma^{\ell}(\i))([\i'] \times [\j']) ,
\end{align*}
where $\tilde \i := \sigma^{\ell}(\i)|_{k-\ell} = i_{\ell} i_{\ell+1}\dots i_{k-1}.$
\item[(2)] Since $k \geq |\j'|$, any $\v \in [\i|_k \i']$ satisfies $\mu_\i[\j'] = \mu_\v[\j']$ by the definition of the conditional measure $\mu_\i$. Thus,
$$\int_{[\i|_k \i']} \mu_\v[\j'] \,d\pi_1 \mu(\v) = \mu_\i[\j'] \pi_1\mu[\i|_k \i'] = \Psi(\i)([\i'] \times [\j']) \pi_1 \mu[\i|_k],$$
proving the claim. As the choice of cylinders $[\i'], [\j']$ was arbitrary the claim of weak convergence follows.
\end{itemize}
\end{proof}

\subsection{Reduction to simple functions}
\label{subsection3}

To obtain Proposition \ref{thm:generatedirrational} we can in fact verify it just for characteristic functions over suitable box-like sets in $\T \times \Xi$. Before giving the exact argument of how to this, let us first introduce some notation to set things up.

\begin{notation}
Fix $h \in \N$ and write
\begin{equation*}\mathcal{P}_h = \Psi(I^\infty) \cup \bigcup_{k\geq h}\Psi_k(I^\infty),\end{equation*} that is, the set of minimeasures of generation at least $h$ together with all the micromeasures of $\mu$. This space is a closed subset of $\mathcal{P}(\Sigma)$ and thus compact.    
Given $\v \in I^\infty$ write
$$B(\v,h) := \{\w\in I^\infty\,:\, \Psi_h(\w) = \Psi_h(\v)\}.$$
Now the sets of interest for us are the push-forwards of $B(\v,h)$ in $\cP(\Sigma)$ defined by
$$\cB(\v,h) := \Psi(B(\v,h)) \cup \bigcup_{k \geq h} \Psi_k(B(\v,h)).$$
\end{notation}

The following lemma shows that the push-forwards $\cB(\v,h)$ are in fact open sets in $\Psi(I^\infty)$ with diameter $m^{-h}$, where $m$ is the size of the slowly expanding direction $I$.

\begin{lemma} \label{openball} There exists $h_0 \in \N$ such that if $h \geq h_0$ and $\v \in I^\infty$, then
$$\cB(\v,h) = B_d(\Psi(\v),m^{-h}) \cap \cP_h,$$
where $B_d$ indicates that this is an open metric ball with respect to $d$, which we recall is the Prokhorov metric.
\end{lemma}

\begin{proof}
Let $h_0 \in \N$ be any integer satisfying
$$m^{-h_0} < \min\{|p_{i}(j) - p_{i'}(j)| : i,i' \in I, j \in J, |p_{i}(j) - p_{i'}(j)| > 0\}.$$
If the set we are taking minimum over is empty, then we just set $h_0 = 0$. 

Pick $\nu \in \cB(\v,h)$. Then $\nu$ is either $\Psi(\w)$ or $\Psi_k(\w)$ for some $k \geq h$ and $\w \in I^\infty$ with $\Psi_h(\w) = \Psi_h(\v)$. By Lemma \ref{miniandmicro} the measures $\Psi_k(\w)$ and $\Psi(\w)$ agree on all cylinders of generation at most $k$. Moreover, since $k \geq h$ and $\Psi_h(\w) = \Psi_h(\v)$ by the choice of $\w$, we have that $\nu$ agrees with $\Psi(\v)$ on all cylinders of generation at most $h$. Now take a Borel-set in $A \subset \Sigma$ and $\epsilon \in (m^{-h-1},m^{-h})$. Then as we used the maximum metric obtained from $I^\infty$ and $J^\infty$, and $m < n$, the $\epsilon$-neighborhood $A_\epsilon$ is by definition a union of $h$ generation cylinder sets $C \subset \Sigma$. Thus
$$\nu(A_\epsilon) = \Psi(\v)(A_\epsilon),$$
which, when inputted to the definition of Prokhorov distance, gives 
$$d(\nu,\Psi(\v)) \leq \epsilon < m^{-h}.$$

Now assume $\nu \in \cP_h$ and $d(\nu,\Psi(\v)) < m^{-h}$. Since $\nu \in \cP_h$, it is either $\Psi(\w)$ or $\Psi_k(\w)$ for some $k \geq h$ and $\w \in I^\infty$. We only need to show that $\Psi_h(\w) = \Psi_h(\v)$. Suppose this is not the case, which means that there is one $l < h$ and $j \in \N$ with $p_{w_l}(j) \neq p_{v_l}(j)$. Write $C = \bigcup_{\j' \in J^{l-1}} [\j' j]$, so $C \subset J^\infty$ is a finite union of cylinders. Since the $\epsilon$-neighborhood of $X \times C$ is $X \times C$ for any $\epsilon < m^{-l}$ and $l < h$, we have 
$$ |\nu(X\times C) - \Psi(\v)(X\times C)| \leq d(\nu,\Psi(\v)) < m^{-h}.$$
Since $k \geq h$ and $\nu$ and $\Psi_h(\w)$ agree on all cylinders of length at most $k$, we obtain
\begin{align*}|p_{w_l}(j)-p_{v_{l}}(j)| = \Big|\sum_{\j' \in J^{l-1}} \mu_\w[\j' j] - \sum_{\j' \in J^{l-1}} \mu_{\v}[\j' j]\Big| = |\nu(X\times C) - \Psi(\v)(X\times C)| < m^{-h},\end{align*}
which contradicts with the choice of $h_0$. Thus $\nu \in \cB(\v,h)$.
\end{proof}

Now assume we have verified the following:

\begin{lemma}\label{simplefunctions}
Fix an interval $[a,b) \subset \T$, finite words $\i', \j'$ in $I^*$ and $J^*$ respectively, a point $\v \in I^\infty$ and $h \geq h_0$ where $h_0$ is defined by Lemma \ref{openball}. Write
\begin{align}\label{simple}g := \chi_{[a,b) \times [\i'] \times [\j'] \times \cB(\v,h)}.\end{align}
If $\alpha$ is irrational, then at $\mu$ almost every $(\i,\j) \in \Sigma$, we have for all $q \in \N$ and $t \in \T$ that
$$\frac{1}{N} \sum_{k = 0}^{N-1} g(Z^{qk}(t,\i,\j,\mu)) \to \int g \, d(\lambda \times P).$$
If $\alpha$ is rational, we can replace $\lambda \times P$ by $\tau_{t,q} \times P$ above.
\end{lemma}

Suppose we have a continuous $f  : \T \times \Xi \to \R$. Since the orbit $Z^k(t,\i,\j,\mu)$ and the measure component $\tilde Q$ of the distribution under study lives on $\Psi(I^\infty)$, we only need to consider the values of $f$ in $\cP_h$ for $k \geq h$. The aim is to uniformly approximate $f$ using functions $\psi$ obtained from functions of the form \eqref{simple}. In other words, given $\epsilon > 0$ we have to construct such $\psi : \T \times \Xi \to \R$ such that when taking the sup-norm
$$\|f|_{\cP_h}-\psi\|_\infty < \epsilon.$$
If one then inserts this into the ergodic averages considered above, we obtain the statement of Proposition \ref{thm:generatedirrational} also for $f$:
$$\frac{1}{N} \sum_{k = 0}^{N-1} f(Z^{qk}(t,\i,\j,\mu)) \to \int f \, d(\lambda \times P)$$
if $\alpha$ is irrational, and if $\alpha$ is rational, then replacing $\lambda \times P$ by $\tau_{t,q} \times P$.  Thanks to Lemma \ref{openball}, we now have all the ingredients required to construct such $\psi$, and the method we use is more or less standard.  Since $\T \times \Xi$ is a product of compact metric spaces, all continuous functions on $\T \times \Xi$ are uniformly continuous, and we can reduce the problem to studying continuous functions on each component of $\T \times \Xi$. In the $\T$ and $\Sigma$ components, it is standard that finite linear combinations of maps of the form $\chi_{[a,b)}$ and $\chi_{[\i'] \times [\j']}$ can be used to obtain such uniform approximation. 

In the measure component, since by Lemma \ref{openball} the sets $\mathcal{B}(\v,h)$ are open, we can use the compactness of $\cP_h$ to find a cover by finitely many balls $\cB(\v,h)$. Moreover, by Lemma \ref{openball} their diameters are always less than $m^{-h}$, so if we choose $h$ large enough the values of a uniformly continuous function $f_0$ on $\cP(\Sigma)$ on each of these balls are $\epsilon$-far from each other. To create suitable simple functions from $\chi_{\cB(\v,h)}$ one can take refinements from $\cB(\v,h)$ to obtain disjoint elements and put a suitable value of $f_0$ on each of these building blocks. 

\subsection{Proof for the simple functions}
\label{subsection4}
We now prove Lemma \ref{simplefunctions}. Fix $t \in \T$. Let $k_1 < k_2 < \dots$ be the sequence of values satisfying
$$\phi^{k_i}(t) \in [a,b).$$
Notice that if $\alpha$ is rational, then it may be that the orbit of $t$ never hits $[a,b)$, in which case such a choice is not possible. However, in this case the proof is done, since then the $\tau_t$ measure of $[a,b)$ is also zero, which gives the proof for simple functions. Thus we may assume, in both cases, that such a sequence $k_i$ can be chosen.

Writing $B = B(\v,h)$, $\cB = \cB(\v,h)$ and $\ell_i = \ell_t(k_i)$ we have by the definition of $g$ and $Z$, and Lemma \ref{openball} that
\begin{eqnarray*} X_i(\i,\j) := g(Z^{k_i}(t,\i ,\j,\mu )) & =&  \chi_{[\i']}(\sigma^{k_i}(\i))\chi_{[\j']}(\sigma^{\ell_i}(\j))\chi_{\cB}(\Psi_{k_i-\ell_i}(\sigma^{\ell_i}(\i))) \\ 
& = &  \chi_{[\i']}(\sigma^{k_i}(\i))\chi_{[\j']}(\sigma^{\ell_i}(\j))\chi_{B}(\sigma^{\ell_i}(\i)).
\end{eqnarray*} 
Consider $X_i$ as a stochastic process in the probability space $(\Sigma,\mu)$. Notice that if $k_i-\ell_i\geq h$ the events $\sigma^{-(k_i-\ell_i)}[i'] \times Y$ and $B \times [j']$ are $\mu$ independent since having the property $\Psi_h(\i) = \Psi_h(\v)$ does not depend on the letters of the word $\i$ after $h$. Thus the process $X_i$ has a common distribution and by $\sigma \times \sigma$ invariance of $\mu$ we obtain
\begin{align*}\E(X_i) &= \mu((\sigma^{-(k_i-\ell_i)}[\i'] \cap B) \times [\j']) = \mu([\i'] \times Y) \mu(B \times [\j']) 
\\ &= \int_B  \pi_1 \mu[\i']\mu_\i[\j'] \, d\pi_1 \mu(\i)= \int g \, dQ =: \rho.
\end{align*}
We define a centred process $\tilde{X}_i = X_i - \rho$.

In order to understand the $\mu$ almost sure behavior of Ces\`aro averages of the sequence of random variables $\tilde{X}_i$ we make use of Lyons' strong law of large numbers for weakly correlated random variables \cite{Lyo88}.

\begin{prop}[Theorem 1 of \cite{Lyo88}]\label{lyons} Let $\tilde{X}_i$ $(i = 1, 2, \dots)$ be random variables with mean zero, $|\tilde{X}_i| \leq 1$ almost surely and
$$\sum_{N = 1}^{\infty}\frac{1}{N} \Big\| \frac{1}{N} \sum_{ i = 1}^{N}\tilde{X}_i \Big\|_2^2 < \infty.$$
Then $\frac{1}{N} \sum_{i = 1}^{N}\tilde{X}_i \to 0$ as $N \to \infty$ almost surely.
\end{prop}

\begin{lemma}\label{satisfylyons}The variables $\tilde{X}_i$ satisfy the hypotheses of Proposition \ref{lyons}.  \label{estimate}\end{lemma}

\begin{proof} We first observe that 
\begin{align}\label{decay} \Big\| \frac{1}{N} \sum_{ i = 1}^{N}\tilde{X}_i \Big\|_2^2 = \frac{1}{N^2} \sum_{1\leq i,j \leq N} \E[\tilde{X}_i \tilde{X}_j] \leq  \frac{1}{N^2}  \sum_{i=1}^N \#\{1 \leq j \leq N \,:\,\E[\tilde{X}_i \tilde{X}_j] \neq 0\}. \end{align}
Thus we need to estimate the number
$$N_{i} =  \#\{1 \leq j \leq N \,:\,\E[\tilde{X}_i \tilde{X}_j] \neq 0\}.$$
For fixed $1\leq i \leq N$ we observe that the value of the function $\tilde{X}_i$ depends only on coordinates 
$$\ell_i,\ell_i+1,\dots,\ell_i+h-1 \quad \text{and} \quad k_i,k_i+1,\dots,k_i+|\i^\prime|-1$$ 
of $\i$ and on coordinates 
$$\ell_i,\ell_i+1,\dots,\ell_i+|\j^\prime|-1$$ 
of $\j$.   If $\E[\tilde{X}_i \tilde{X}_j] \neq 0$, then $X_i$ and $X_j$ must not be independent, which can only happen if one of the following four situations occurs:
\begin{itemize}
\item[(1)] $|\ell_i - \ell_j| \leq \max\{h,|\j^\prime|\} $
\item[(2)] $|k_i - \ell_j| \leq \max\{h, |\i^\prime|,|\j^\prime|\}$
\item[(3)] $|\ell_i - k_j| \leq \max\{h, |\i^\prime|,|\j^\prime|\}$
\item[(4)] $|k_i - k_j| \leq |\i^\prime|.$
\end{itemize}
Thus if we write $C := \max\{h, |\i^\prime|,|\j^\prime|\}$, we have a bound
\begin{align*}
N_i & \leq  \#\{1 \leq j \leq N \,:\,|\ell_i - \ell_j| \leq  C \}  +   \#\{1 \leq j \leq N \,:\,|k_i - \ell_j| \leq  C \} \\
& \quad +   \#\{1 \leq j \leq N \,:\,|\ell_i - k_j| \leq  C \}   +   \#\{1 \leq j \leq N \,:\,|k_i - k_j| \leq  C \}.
\end{align*}
Recall that the $k_i$ form a strictly increasing sequence of integers and observe that the $l_i$ form an increasing sequence of integers with at most $1/\alpha$ repetitions of any particular integer.  It follows that each of the four terms above can be bounded by $(2C+1)/\alpha$ which yields $N_i \leq 4(2C+1)/\alpha$. Hence
\[
\sum_{N = 1}^{\infty}\frac{1}{N} \Big\| \frac{1}{N} \sum_{ i = 1}^{N}\tilde{X}_i \Big\|_2^2     \leq   \sum_{N = 1}^{\infty} \frac{4(2C+1)/\alpha }{N^2}  <  \infty,
\]
which completes the proof.
\end{proof}

\begin{proof}[Proof of Lemma \ref{simplefunctions}]
Initially, set $q=1$.  First we observe that
\begin{equation}\label{eq:link}\frac{1}{N} \sum_{k=0}^{N-1} g(Z^k(t,\i,\j,\mu))  =  \frac{1}{N} \sum_{i \,:\, k_i<N} X_i .\end{equation}
Combining Proposition \ref{lyons} and Lemma \ref{estimate} yields
\begin{equation*}\lim_{N\to\infty} \frac{1}{ \#\{0 \leq k < N\,:\, \phi^k(t) \in [a,b)\} }  \sum_{i \,:\, k_i<N} X_i = \int g \, dQ.
\end{equation*}
By the equidistribution of $\{\phi^k(t)\}_{k \in \N}$ to Lebesgue measure for irrational rotations and to $\tau_{t,1}$ for rational rotations, the averages
$$ \frac{\#\{0 \leq k < N\,:\, \phi^k(t) \in [a,b)\}}{N}$$
converge to $\lambda[a,b)$ in the irrational case and to $\tau_{t,1}[a,b)$ in the rational case. Thus the ergodic averages
\begin{equation*}\lim_{N\to\infty} \frac{1}{N} \sum_{k=0}^{N-1} g(Z^k(t,\i,\j,\mu))\end{equation*}
converge to either $\lambda[a,b)\int g \, dP$ in the irrational case or to $\tau_{t,1}[a,b)\int g \, dP$ in the rational case.
This shows (\ref{eq:stronggen}) in the case $q=1$.  The case when $q>1$ is similar and we omit the details.  First observe that the measure $\mu$ is still a Bernoulli measure for $\sigma^q \times \sigma^q$.  In the irrational case the sequence $\{\phi^{qk}(t)\}_{k \in \mathbb{N}}$ still equidistributes to Lebesgue measure and so the proof is identical.  In the rational case when $\alpha = p'/q'$ with $\text{gcd}(p',q') = 1$, say, then $\{\phi^{qk}(t)\}_{k \in \mathbb{N}}$ equidistributes to $\tau_{t,q}$, which differs from $\tau_{t,1}$ if $\text{gcd}(q,q')>1$.
\end{proof}

\subsection{Invariance and ergodicity of the distributions}
\label{subsection5}

\begin{lemma}
The distributions $\lambda \times P$ and $\tau_{t,q} \times P$ are $Z$ invariant and $Z$ ergodic.
\end{lemma}
\begin{proof}
We prove this just for the irrational case.  The proof in the rational case is exactly the same as we do not use any properties of Lebesgue measure apart from $\phi$ invariance and $\phi$ ergodicity, which is true in the rational case as well for $\tau_{t,q}$.

Let us first prove the $Z$ invariance of $\lambda \times P$. Fix words $\i',\v \in I^k$ and $\j' \in J^l$ and an interval $[a,b] \subset \T$. Define the set
$$A := [a,b] \times [\i'] \times [\j'] \times \Psi([\v])$$
It is enough to show that
$$(\lambda \times P)(Z^{-1} A) = (\lambda \times P)(A).$$
Given $t \in \T$ write
$$A_t = \{(\i,\j,\nu) \in \spt P : Z(t,\i,\j,\nu) \in A\}.$$
Now we can decompose
$$(\lambda \times P)(Z^{-1} A) = \int\limits_{(\phi^{-1}[a,b]) \cap [1-\alpha,1)} P(A_t) \, dt + \int\limits_{(\phi^{-1}[a,b]) \setminus [1-\alpha,1)} P(A_t) \, dt.$$
Then by the definition of $P$ and $\Psi$ we have
$$A_t = \begin{cases}\sigma^{-1}[\i'] \times \sigma^{-1}[\j'] \times \Psi(\sigma^{-1}[\v]), & \text{ if } t \in [1-\alpha,1);\\
\sigma^{-1}[\i'] \times [\j'] \times \Psi([\v]), & \text{ if } t \in [0,1-\alpha).
\end{cases}$$
Hence by the $\sigma$ invariance of $\pi_1 \mu$ and the fact that $\sigma\mu_\i = \mu_{\sigma(\i)}$ for all $\i \in I^{\infty}$, we have that when $t \in [1-\alpha,1)$ the measure
$$P(A_t) = \int_{\sigma^{-1}[\v]} \pi_1 \mu(\sigma^{-1}[\i']) \times \mu_\i(\sigma^{-1}[\j']) \, d\pi_1 \mu(\i) = \int_{[\v]} \pi_1 \mu[\i'] \times \mu_\i[\j'] \, d\pi_1\mu(\i).$$
Moreover, if $t \in [0,1-\alpha)$ we have just using the $\sigma$ invariance of $\pi_1 \mu$ that
$$P(A_t) = \int_{[\v]} \pi_1 \mu(\sigma^{-1}[\i']) \times \mu_\i([\j']) \, d\pi_1 \mu(\i) = \int_{[\v]} \pi_1 \mu[\i'] \times \mu_\i[\j'] \, d\pi_1\mu(\i).$$
Hence by combining the previous two expressions we have by the $\phi$ invariance of $dt$ that
$$(\lambda \times P)(Z^{-1} A) = \int_{\phi^{-1}[a,b]} \int_{[\v]} \pi_1 \mu[\i'] \times \mu_\i[\j'] \, d\pi_1\mu(\i) \, dt = (\lambda \times P)(A)$$
as claimed.

We will now show ergodicity. We use \cite[Theorem 1.17(i)]{walters}, which guarantees that it is enough to prove that
\begin{equation*}
\lim_{N\to\infty} \frac{1}{N} \sum_{k=0}^{N-1} (\lambda \times P)(A_1\times B_1 \cap Z^{-k} (A_2\times B_2  )) = (\lambda \times P)(A_1\times B_1) (\lambda \times P)(A_2\times B_2).
\end{equation*}
for all open sets $A_1,A_2$ in $\T$ and cylinder sets
$$B_1 = [\i^{(1)}]\times [\j^{(1)}] \times \Psi[\v^{(1)}] \quad \text{and} \quad B_2 = [\i^{(2)}]\times [\j^{(2)}] \times \Psi[\v^{(2)}]$$
for $\i^{(1)}, \i^{(2)}, \v^{(1)}, \v^{(2)} \in I^*$ and $\j^{(1)}, \j^{(2)} \in J^*$. In particular, to see why the sets $\Psi[\v^{(1)}]$ generate the Borel-sets in $\Psi(I^\infty)$, simply observe that every open set in $\Psi(I^\infty)$ can be obtained as the union of images of cylinders in $I^\infty$.  We first observe that if $\tilde \i \in I^\infty$, then
\begin{equation*}Z(t,\i,\j , \pi_1\mu \times \mu_{\tilde{\i}}) = \begin{cases} (\phi(t), \sigma(\i) , \j, \pi_1\mu \times \mu_{\tilde{\i}}) ,& \text{if } t+\alpha < 1;\\
(\phi(t), \sigma(\i) , \sigma(\j), \pi_1\mu \times \mu_{\sigma(\tilde{\i})}), & \text{otherwise}.
\end{cases}\end{equation*}
Defining a map $\tilde{Z}_{t,k} : \Sigma \times \Psi(I^\infty) \to \Sigma \times \Psi(I^\infty)$ by
$$
\tilde{Z}_{t,k}(\i,\j,\pi_1\mu\times \mu_{\tilde{\i}}) = ( \sigma^k(\i), \sigma^{\ell_k(t)}(\j), \pi_1\mu \times \mu_{\sigma^{\ell_k(t)}(\tilde{\i})}),
$$
we thus have
\begin{equation*}Z^k(t,\i,\j , \pi_1\mu \times \mu_{\tilde{\i}}) = (\phi^k(t), \tilde{Z}_{t,k}(\i,\j,\pi_1\mu\times \mu_{\tilde{\i}})). \end{equation*}
We claim that for sufficiently large $k$, we have for any $t\in \T$
\begin{equation}\label{eq:strongmixing}
 P(B_1\times  \tilde{Z}^{-1}_{t,k} (B_2))=P(B_1)P(B_2).
\end{equation}
Indeed, for any $k$ we have that
$$B_1\cap \tilde{Z}^{-1}_{t,k}(B_2)  =  C_1 \times C_2 \times \Psi(C_3)$$
where
$$C_1 =  [\i^{(1)}] \cap \sigma^{-k} [\i^{(2)} ], \quad C_2 = [\j^{(1)} ] \cap \sigma^{-\ell_k(t)} [\j^{(2)} ], \quad \text{and}\quad  C_3 = [\v^{(1)} ] \cap \sigma^{-\ell_k(t)} [\v^{(2)} ].$$
Thus choosing $k$ large enough so that $k \geq | \i^{(1)} |$ and $\ell_t(k) \geq \max \{ | \j^{(1)} |, | \v^{(1)} |\}$, we can use independence and $\sigma$ invariance of $\pi_1 \mu$ to obtain
\begin{align*}P(B_1\cap \tilde{Z}^{-1}_{t,k}(B_2) ) & =  \int_{\Psi^{-1} \Psi C_3}\pi_1 \mu(C_1)  \mu_{\i^\prime} ( C_2) \, d\pi_1\mu(\i^\prime) \\
& =  \pi_1 \mu  [\i^{(1)}]  \pi_1 \mu [\i^{(2)}] \int\limits_{ \Psi^{-1} \Psi C_3  } \mu_{\i'} [\j^{(1)} ] \mu_{\i'} ( \sigma^{-\ell_k(t)} [\j^{(2)} ] ) \,d\pi_1 \mu(\i') \\
& =   \pi_1 \mu  [\i^{(1)}]  \pi_1 \mu [\i^{(2)}] \int\limits_{ \Psi^{-1}  \Psi [\v^{(1)} ]  } \mu_{\i'} [\j^{(1)} ]  \,d\pi_1 \mu(\i') \int\limits_{ \Psi^{-1} \Psi [\v^{(2)}]  }\mu_{\i'}  [\j^{(2)} ] \,d\pi_1 \mu(\i')  \\
& =  P(B_1)P(B_2).
\end{align*}
Note that the map $\Psi$ is not necessarily injective so $\Psi^{-1} \Psi C_3$ may not be equal to $C_3$, so we have to integrate over the pre-images. Next, unpacking the definition of $P$ and using (\ref{eq:strongmixing}) we see that for large enough $k$ the measure
\begin{eqnarray*} (\lambda \times P)(A_1\times B_1 \cap Z^{-k} (A_2\times B_2  )) & =&  \int \chi_{A_1 \cap \phi^{-k}(A_2) } (t) \,P(B_1\times  \tilde{Z}^{-1}_{t,k} (B_2)) \,d\lambda(t)\\
& = & \lambda(A_1 \cap \phi^{-k}(A_2) ) P(B_1)P(B_2) . \end{eqnarray*}
Summing over $k$ and using $\phi$ ergodicity of $\lambda$ completes the proof of ergodicity.
\end{proof}

\section{Proof of the projection theorem}
\label{sec:proj}

In this section we prove Theorem \ref{thm:proj}. We suppose $\alpha = \log m / \log n$ is irrational. Let $Q$ be our CP-chain constructed previously and recall from Proposition \ref{semicproj} the definition of $E : \Pi_{2,1} \to \R$ which has the form
$$E(\pi) = \iint \dim \pi [S_t(\pi_1 \mu \times \mu_x)] \, d\pi_1 \mu (x) \, dt$$
since $\alpha$ is irrational.

\begin{lemma}
\label{invariant}
For any $\pi \in \Pi_{2,1} \setminus \{\pi_1,\pi_2\}$ we have 
$$E(\pi) = \min\{1,\dim \mu\}.$$ 
\end{lemma}

\begin{proof}
Parametrise non-vertical and non-horizontal orthogonal projections in $\Pi_{2,1}$ by projections $\pi_s : \R^2 \to \R$, $s \in \R$, by
$$\pi_s^\pm(x,y) = x \pm n^{-s} y, \quad (x,y) \in \R^2.$$
Let us just prove the result for $\pi_s = \pi_s^+$, as the case for $\pi_s^-$ is symmetric. Fix $s \in \R$. We have for every $(x,y) \in \R^2$ that
$$\pi_{s} \circ S_t(x,y) = \pi_s(n^{-t/2}x,n^{t/2}y) = n^{-t/2}x + n^{t/2-s}y = n^{-t/2}\pi_{s-t}(x,y).$$
Thus as $\dim$ does not change under scaling by $n^{-t/2}$, we have for all $x$ and $t$ that
$$\dim \pi_{s} S_t(\pi_1 \mu \times \mu_x) = \dim \pi_{s-t}(\pi_1 \mu \times \mu_x).$$
Marstrand's projection theorem (the measure version presented by Hunt and Kaloshin in \cite[Theorem 4.1]{HuntKaloshin1997}) implies that $\dim \pi_{s-t}(\pi_1 \mu \times \mu_x) = \min\{1,\dim(\pi_1 \mu \times \mu_x)\}$ for Lebesgue almost every $t \in [0,1)$ and by the definition of our CP-chain, $\dim (\pi_1 \mu \times \mu_x) = \dim \mu$ for $\pi_{1} \mu$ almost every $x$.  It follows that
$$E(\pi_s) = \iint \dim \pi_{s-t}(\pi_1 \mu \times \mu_x) \, dt \, d\pi_1 \mu (x) = \min\{1,\dim \mu\},$$
which is the claim.
\end{proof}

Lemma \ref{invariant} proves Theorem \ref{thm:proj} since by Proposition \ref{semicproj}(3) we have $\dim \pi \mu \geq E(\pi)$ and projections do not increase $\dim$.

\section{Proofs of the distance set results}
\label{sec:dist}

We first prove our general theorem which gives an ergodic theoretic criterion for proving the dimension version of the distance set conjecture.  We then apply this theorem in two concrete settings.  As a first application we verify the conjecture for self-affine carpets in the Bedford-McMullen class, provided the dimension is greater than or equal to 1, and the Lalley-Gatzouras and Bara\'nski classes, provided the dimension is strictly greater than 1, proving Corollary \ref{cor:selfaffinedist}. Then we show how our technique may be used to recover results of Orponen \cite{orponen} and Falconer-Jin \cite{jin}, proving Corollary \ref{cor:selfsimdist}.  Finally we prove Proposition \ref{thm:distset2} which studies the conjecture in the case when we know the exceptional set in Marstrand's projection theorem is small and apply it to irrational type carpets.

\subsection{Proof of Theorem \ref{thm:distset}} Let $Q$ be the CP-chain which $\mu$ generates and write $K = \spt \mu$. We may assume $K$ is a compact uncountable subset of $[0,1]^2$.  The \emph{direction set} of $K$ is
\[
S(K) = \bigg\{\frac{x-y}{\lvert x-y\rvert} : x,y \in K, \, x \neq y \bigg\} \subset S^1.
\]
We may assume that $S(K)$ is dense in $S^1$. This is because of the following observation of Orponen already given in \cite{orponen}. Namely, if $S(K)$ is not dense in $S^1$, then $K$ is contained in a rectifiable curve by \cite[Lemma 15.13]{Mattila95}. Then as $\cH^1(K) > 0$, a result of Besicovitch and Miller \cite{besmil} gives that $D(K)$ contains an interval, and in particular $\dim D(K) = 1$.

Let $E$ be the lower semicontinuous function given by $Q$, defined above in Proposition \ref{semicproj}, and let $\epsilon > 0$. By the lower semicontinuity of $E$, Proposition \ref{semicproj} gives us that the set $\Pi_\epsilon \subset \Pi_{2,1}$ consisting of `good projections', i.e. those $\pi$ satisfying 
$$E(\pi) > \min\{ 1, \dim \mu \} - \epsilon$$ 
is open and non-empty (in fact it is also dense and of full measure).  Identifying $\Pi_{2,1}$ with the upper half of $S^1$ and using denseness of $S(K)$ and openness of $\Pi_\epsilon$, choose two points $x,y \in K$ such that the direction
\[
\frac{x-y}{\lvert x-y\rvert} \equiv \pi \in \Pi_\epsilon.
\]
Let $g: [0,1]^2 \setminus B(x, \lvert x-y\rvert/3)  \to \mathbb{R}$ be the distance map defined by $g(z) = \lvert z-x \rvert$. Notice that $g$ is $C^1$ and the derivative $D_z g = (z-x)/|z-x|$ in the support of $g$. Now extend $g$ to a $C^1$ mapping $g$ on the whole of $[0,1]^2 $. Fix $q \in \mathbb{N}$. By the openness of $\Pi_\epsilon$, we may choose $r>0$ small enough such that
\begin{itemize}
\item[(1)] $r$ is sufficiently small compared to the distance between $x$ and $y$, i.e. $r < \lvert x-y\rvert/3$
\item[(2)] the derivative of $g$ is sufficiently close to $\pi$ on $B(y,r)$, i.e. the norm
\begin{align}\label{forg}
\sup_{z \in B(y,r)} \| D_z g - \pi\| < \delta^q,
\end{align}
where $\delta$ is the constant coming from the definition of the $\delta$-regular partition operator.
\end{itemize}
Consider the restricted and normalised measure 
$$\nu = \mu(B(y,r))^{-1}\mu\vert_{B(y,r)}.$$  
It is a consequence of the Besicovitch density point theorem that $\nu$ generates the \textit{same} CP-chain $Q$ at $\nu$ almost every $z$; see for example \cite[Lemma 7.3]{HocShm12}.   Proposition \ref{prop8.4} now gives us that
$$\dim g \nu \geq E_q(\pi) - C/q,$$
where $C$ depends only on the fixed regularity parameter $\delta \in (0,1)$ of the partition operator associated to the CP-chain $Q$.  Since $g$ maps $B(y,r) \cap K$ into $D(K)$, we have that $g\nu$ is supported on $D(K)$ and since $q \in \mathbb{N}$ was arbitrary we can pass to the limit giving
$$\dim D(K) \geq \dim g  \nu \geq E_q(\pi) - C/q \to E(\pi)$$
As $\pi \in \Pi_\epsilon$, we obtain  $\dim D(K) > \min\{1,\dim \mu\}-\epsilon$ completing the proof as $\epsilon > 0$ was arbitrary.

\subsection{Proof of Corollary \ref{dyadicmicro}}
\hspace{1mm}
We may assume that the measure $\mu$ is supported on $[0,1/2]^2$ by rescaling the original measure to there. Applying a random translation argument, Hochman and Shmerkin exhibited in \cite[Theorem 7.10]{HocShm12} that for Lebesgue almost every $\omega \in [0,1/2]^2$ we have that at $\mu$ almost every $x$ there is an ergodic CP-chain $Q$ for the dyadic partition operator supported on dyadic micromeasures of $\mu+\omega$ at $x+\omega$ of dimension $\dim Q$ at least $\dim \mu$. Here $\mu+\omega$ is the translate of $\mu$, defined by $\mu(A-\omega)$ for $A \subset \R^2$. In fact, the dimension of the CP-chain is at least the so called \textit{upper entropy dimension} of the measure $\mu$, which is greater than or equal to $\dim \mu$. On the other hand, given an ergodic CP-chain $Q$, then $Q$ almost every measure $\nu$ is exact dimensional with dimension $\dim \mu$ and generates $Q$. As $\dim \nu = \dim \mu > 1$, we have $\cH^1(\spt \nu) > 0$, so we may apply Theorem \ref{thm:distset} to any one of these $\nu$ to obtain the claim.

\subsection{Proof of Corollary \ref{cor:selfaffinedist}}

Proving Corollary \ref{cor:selfaffinedist} for Bedford-McMullen carpets follows easily from Theorem \ref{thm:distset}, combined with the fact that if a Bedford-McMullen carpet $K$ has $\dim K \geq 1$, then $\mathcal{H}^1(K)>0$, see \cite{perescarpets}.  We can find a Bernoulli measure $\mu$ supported on $K$ with full Hausdorff dimension - commonly referred to as the McMullen measure, see \cite{mcmullen}. Thus combining Theorem \ref{thm:main} and Theorem \ref{thm:distset} it is immediate that such a Bedford-McMullen carpet satisfies the distance set conjecture.  To extend this to the more general classes of self-affine carpet described in Corollary \ref{cor:selfaffinedist} we rely on the following result due to Ferguson-Jordan-Shmerkin \cite{fjs}:

\begin{lemma}[Lemma 4.3 of \cite{fjs}] \label{FJS}
Given any Lalley-Gatzouras or Bara\'nski type carpet $K$ and $\varepsilon>0$, there exists a self-affine set $K_\varepsilon \subset K$ with
\[
\dim K_\varepsilon \geq \dim K - \varepsilon
\]
which is the attractor of an IFS of maps on $[0,1]^2$ with linear part of the form
\[
\begin{pmatrix}
a & 0\\
0 & b
 \end{pmatrix}
\]
for fixed uniform constants $a,b \in (0,1)$.  Moreover, this subsystem has uniform fibers in the sense that when the rectangles corresponding to the maps are projected in the direction of the fastest expansion, they are either disjoint or lie perfectly on top of each other and there are the same number of rectangles in each of the resulting columns.
\end{lemma}

In light of this lemma, in order to prove Corollary \ref{cor:selfaffinedist} for Lalley-Gatzouras or Bara\'nski type carpets $K$ with $\dim K >1$, it suffices to prove it for self-affine sets of the form described in the lemma.  In essence, these sets are equivalent to Bedford-McMullen carpets.  Of course, they need not be $T_{m,n}$ invariant and $a,b$ need not be the reciprocals of integers $m,n$, but they share many of the same properties.  Indeed the analogue of the McMullen measure is a product measure easily seen to have full Hausdorff dimension and, more importantly, the symbolic dynamics used to describe them are identical to the Bedford-McMullen case and thus we can construct an ergodic CP-chain in exactly the same way as we did when proving Theorem \ref{thm:main} and this, combined with Theorem \ref{thm:distset}, proves the result.

\subsection{Proof of Corollary \ref{cor:selfsimdist}}

Proving the distance set conjecture for arbitrary planar self-similar sets with positive length follows almost immediately from Theorem \ref{thm:distset} when we observe the following.

\begin{lemma}\label{selfsimcpchain}
For any self-similar measure $\mu$ in $\mathbb{R}^2$ satisfying the strong separation condition (SSC), there is an ergodic CP-chain $Q$ for the dyadic partition operator supported on the dyadic micromeasures of $\mu$ with dimension at least $\dim \mu$ and such that the dyadic micromeasures $\nu$ are of the form 
$$\nu = \mu(B)^{-1}S(\mu|_B)$$
for some Borel-set $B$ with $\mu(B) > 0$ and some similitude $S$ of $\R^2$. Moreover, the original measure can be recovered from a given micromeasure $\nu$ as $\mu = \nu(B')^{-1}S'(\nu|_B')$, for some Borel-set $B'$ and similitude $S'$.
\end{lemma}
This follows by combining \cite[Theorem 7.10]{HocShm12}, \cite[Proposition 9.1]{HocShm12} and the subsequent discussion.   Moreover, by \cite[Proposition 7.7]{HocShm12} a $Q$ typical measure $\nu$ generates $Q$ and has dimension $\dim \mu$.  Let $K$ be a self-similar set satisfying the SSC, $\mu$ be a self-similar measure supported on $K$ with full Hausdorff dimension and fix such a $Q$ typical micromeasure $\nu$ and observe that, using the notation from Lemma \ref{selfsimcpchain},
\[
K \supseteq S( \spt \nu ) \qquad \text{and} \qquad \spt \nu \supseteq S'(K)
\]
and therefore the assumption $\mathcal{H}^1(K)>0$ carries through to give $\mathcal{H}^1(\spt \nu)>0$ and
\[
D(K)  \supset D( S(\spt \nu) ) = r  D( \spt \nu ),
\]
where $r > 0$ is the similarity ratio of $S$.  Thus we may apply Theorem \ref{thm:distset} to obtain that
$$\dim D(K)  \geq \dim D(\spt \nu)  \geq \min\{1,\dim \nu\} = 1,$$ 
which gives Corollary \ref{cor:selfsimdist} when $K$ satisfies the SSC. To extend this to the overlapping case we apply the following lemma, which follows from \cite[Proposition 1.8]{farkas}.

\begin{lemma}\label{sscsubsystem}
Given any self-similar set $K$ in the plane with $\mathcal{H}^1(K)>0$ and $\varepsilon>0$, there exists a self-similar set $K_\varepsilon \subset K$ such that $K_\varepsilon$ satisfies the SSC, $\dim K_\varepsilon \geq \dim K - \varepsilon$ and $\mathcal{H}^1(K_\varepsilon)>0$.
\end{lemma}

If we do not require the $\mathcal{H}^1(K_\varepsilon)>0$ condition, then the proof of this (essentially folklore) result follows easily from the Vitali covering lemma.  The details have been written down in, for example \cite[Lemma 3.4]{orponen} in the planar case or \cite[Lemma 1.7]{farkas} which contains the result in higher dimensions, along with a stronger result concerning the rotation groups.  Adding the $\mathcal{H}^1(K_\varepsilon)>0$ condition makes the result harder to prove and relies on delicate exhaustion type arguments.  This was dealt with by Farkas \cite[Proposition 1.8]{farkas}.

\subsection{A technical proposition and proof of Corollary \ref{thm:distsetirrational}} \label{technicalprop} 

Here we state and prove Proposition \ref{thm:distset2}, which was alluded to at the end of Section \ref{sec:intro}.  It gives more precise information about the distance set conjecture in the case when $\dim K \leq 1$.  Finally we apply this proposition to give Corollary \ref{thm:distsetirrational}.

\begin{prop} \label{thm:distset2}
Let $\mu$ be a measure on $\R^2$ which generates an ergodic CP-chain such that
$$E(\pi) = \min\{1,\dim \mu\}$$
at all except countably many $\pi \in \Pi_{2,1}$. Then
\[
\dim D(\spt \mu) \geq \min\{ 1, \dim \mu\}.
\]
\end{prop}

\begin{proof}
The proof follows a similar pattern as the proof of Theorem \ref{thm:distset} and we use the same notation. We may again assume $K = \spt \mu$ is a compact uncountable subset of $[0,1]^2$. We split the proof into two cases depending on whether the direction set $S(K)$ is countable or not.

\begin{itemize}
\item[(1)] Suppose $S(K)$ is countable. In this case $K$ must be contained in a union of countably many lines and we claim that, in fact, $K$ is contained in a single line.  If it were not, then we would be able to find a line $L$ with $K \cap L$ uncountable and a point $x \in K \setminus L$.  Therefore the set of directions corresponding to $x$ and points from $K \cap L$ is uncountable, which is a contradiction.  Hence essentially we have $K \subset \R$ and it was already observed by Falconer \cite[Theorem 3.1]{falc} that in this case $\dim D(K) = \dim K$.
\item[(2)] Suppose $S(K)$ is uncountable and fix $\epsilon > 0$. By the lower semicontinuity of $E$, we again find that the set $\Pi_\epsilon \subset \Pi_{2,1}$ consisting of `good projections', i.e. those $\pi$ satisfying 
$$E(\pi) > \min\{ 1, \dim \mu \} - \epsilon$$ 
is open and dense.  Moreover, since $S(K)$ is uncountable and by our principal assumption $ \Pi_{2,1} \setminus \Pi_\epsilon$ is at most countable, we can choose $x,y \in K$ such that the projection $\pi \in \Pi_{2,1}$ determined by
$$\pi = \frac{x-y}{|x-y|}$$
lies in $\Pi_\epsilon$.  We can now complete the proof in the same way as for Theorem \ref{thm:distset}.
\end{itemize}
\end{proof}

Finally, Corollary \ref{thm:distsetirrational} follows from Proposition \ref{thm:distset2} since $E(\pi)$ is equal to $\min\{1, \dim \mu\}$ for all $\pi \in  \Pi_{2,1} \setminus \{\pi_1, \pi_2\}$, see the end of Section \ref{sec:proj}.


\section{Prospects}
\label{sec:prospects}

\subsection{General $(\times m,\times n)$ invariant sets and measures}

A natural direction for our work to go in would be to study general $T_{m,n}$ invariant measures.

\begin{conj}\label{projconj}
Suppose $\log m / \log n$ is irrational. Let $\mu$ be a $T_{m,n}$ invariant measure.  Then
$$\dim \pi \mu = \min\{1, \dim \mu\}$$
for every $\pi \in \Pi_{2,1} \setminus \{\pi_1,\pi_2\}$.
\end{conj}

It would be interesting to attempt to extend our results to Gibbs measures for $T_{m,n}$, although we foresee several technical difficulties, not least in the loss of our neat parametrisation of micromeasures.  However, perhaps assuming suitable mixing on the fibers, one could say something.

The set theoretical version of Conjecture \ref{projconj} could be interesting as well. By \cite{fjs} we already know that Bedford-McMullen carpets give examples of $T_{m,n}$ invariant sets that satisfy the projection theorem. However, to approach the problem for more general invariant sets, one may still be able to use Bernoulli measures for $T_{m,n}$. Kenyon and Peres proved in \cite{kenyonperes}, that for any closed $T_{m,n}$ invariant set $X$ we can approximate it from above with nested Bedford-McMullen carpets $C_k$, $k \in \N$, such that the dimensions of the carpets converge to the dimension of the $T_{m,n}$ invariant set. For each such carpet $C_k$, there is the natural measure of full dimension, which is a Bernoulli measure $\mu_k$ on blocks of length $k$ and hence $T_{m,n}^k$-invariant. Defining the measures $\nu_k = \frac{1}{k} \sum_{i=0}^{k-1} T_{m,n}^i \mu_k$ we recover a $T_{m,n}$-invariant measure.  Moreover, the measures $\nu_k$ converge weakly to a measure $\nu$, which is $T_{m,n}$ ergodic, invariant and supported on $X$ of dimension $\dim X$. It may be possible to try to study the scaling scenery of $\nu$ using the scaling sceneries of $C_k$ and thus try to prove a projection theorem for $\nu$.

\subsection{Scaling scenery of other non-conformal constructions}

An interesting direction to pursue would be to study the scaling sceneries and CP-chains generated by other measures with a non-conformal structure. For example the Bernoulli measures on the carpets of Bara\'nski type could be a tractable problem as they have a clear partition operator associated to them.  However, their tangent structure can be quite wild in general, as it may happen that none of the directions dominate the other.  In particular, the original measure can appear as a micromeasure, which is in stark contrast to the typical Bedford-McMullen situation.  Tangent sets for Bara\'nski carpets have been studied in \cite{fraserassouad} in the context of computing the Assouad dimension, but only the extreme cases were considered, where the tangent structure was again a product of a projection and a fiber.  

It would also be natural to consider more general self-affine sets such as the carpets introduced by Feng-Wang \cite{fengaffine} and Fraser \cite{fraseraffine}, but the situation is more complicated there as there is no obvious analogue of approximate squares because of a lack of a grid like structure.  As such the partition operator would be more complicated to define.

\ack{We thank Thomas Jordan and Pablo Shmerkin for stimulating discussions and helpful comments, and Tuomas Orponen for pointing out an error in the distance set part of an earlier version of the manuscript.  Finally, we thank an anonymous referee for their careful reading of the manuscript and for making several helpful suggestions on exposition, clarity and accuracy.}

\bibliographystyle{abbrv}
\bibliography{CPchains}

\end{document}